\documentclass[11pt]{amsart}
\usepackage{amsmath}
\usepackage{amscd}
\usepackage{tikz-cd}
\usepackage[utf8]{inputenc}
\usepackage{mathrsfs} 
\usepackage{amssymb}
\usepackage{bm}
\usepackage{graphicx}
\usepackage[centertags]{amsmath}
\usepackage{amsfonts}
\usepackage{amsthm}
\usepackage{enumerate}
\usepackage{tocvsec2}
\usepackage{xcolor}
\usepackage[margin=1in]{geometry}

\newtheorem{theorem}{Theorem}[section]
\newtheorem*{theorem*}{Theorem}
\newtheorem{xrem}{Remark}
\newtheorem{corollary}[theorem]{Corollary}
\newtheorem{lemma}[theorem]{Lemma}

\newtheorem{proposition}[theorem]{Proposition}

\newtheorem{fact}[theorem]{Fact}
\newtheorem{problem}{Problem}[section]

\theoremstyle{definition}
\newtheorem{definition}[theorem]{Definition}
\newenvironment{remark}[1][Remark]{\begin{trivlist}
\item[\hskip \labelsep {\bfseries #1}]}{\end{trivlist}}

\newcommand{\ee}{\varepsilon}
\newcommand{\nn}{\mathbb{N}}
\newcommand{\rr}{\mathbb{R}}

\newcommand{\sss}{\mathcal{S}}

\begin{document}

\title[Genericity and Universality for Operator Ideals]
{Genericity and Universality for Operator Ideals}
\author[Kevin Beanland and Ryan M. Causey]
{Kevin Beanland and Ryan M. Causey}
\address [Kevin Beanland] {Department of Mathematics, Washington and Lee University, Lexington, VA 24450.} \email{beanlandk@wlu.edu}\
\address [Ryan M. Causey]{Department of Mathematics, Miami University of Ohio} \email{}




\begin{abstract}
A bounded linear operator $U$ between Banach spaces is universal for the complement of some operator ideal $\mathfrak{J}$ if it is a member of the complement and it factors through every element of the complement of $\mathfrak{J}$. In the first part of this paper, we produce new universal operators for the complements of several ideals and give examples of ideals whose complements do not admit such operators. In the second part of the paper, we use Descriptive Set Theory to study operator ideals. After restricting attention to operators between separable Banach spaces, we call an operator ideal $\mathfrak{J}$ generic if anytime an operator $A$ has the property that every operator in $\mathfrak{J}$ factors through a restriction of $A$, then every operator between separable Banach spaces factors through a restriction of $A$. We prove that many of classical operator ideals are generic (i.e. strictly singular, weakly compact, Banach-Saks) and give a sufficient condition, based on the complexity of the ideal, for when the complement does not admit a universal operator. Another result is a new proof of a theorem of M. Girardi and W.B. Johnson which states that there is no universal operator for the complement of the ideal of completely continuous operators.  
\end{abstract}

\maketitle
 \begin{center}
    \emph{In memory of Joe Diestel (1943 - 2017)}
\end{center}

\section{Introduction}

The study of operator ideals, as initiated by A. Pietsch \cite{Pie-book}, has played a prominent role in the development of the theory of Banach spaces and operators on them. To quote J. Diestel, A. Jarchow and A. Pietsch in \cite{DJP-Handbook}:

\begin{center}
   ``\emph{the language of operator ideals offers manifold opportunities to describe specific properties of Banach spaces, giving this landscape a much more colourful view.}" 
\end{center}

The current paper contains several new results related to operator ideals and well as a new framework, based on Descriptive Set Theory, for studying operator ideals. 
Classical results of J. Lindenstrauss and A. Pe\l czy\'{n}ski \cite{LindPel-Studia} and W.B. Johnson \cite{Jo-Colloquium} state that the summing operator $(a_i)_{i=1}^\infty \mapsto (\sum_{i=1}^n a_i)_{n=1}^\infty$ from $\ell_1 \to \ell_\infty$ and the formal identity from $\ell_1 \to \ell_\infty$, are universal for the complement of the operator ideals of weakly compact and compact operators, respectively. Recall that a bounded linear operator between Banach spaces $U$ is universal for the complement of some operator ideal $\mathfrak{J}$ if $U$ factors through every element of the complement of $\mathfrak{J}$. One of the objectives of the current paper is to prove the existence or non-existence of a universal operators for the complements of several operator ideals. The scope of this paper, however, is not limited to these investigations. Indeed, a new observation we make, that we believe is of interest, relates the existence of a universal operator for the complement of an ideal to the Borel complexity of the ideal as a subset of the standard Borel space of operators between separable Banach spaces. All of these terms are explained in section 5. After setting up the descriptive set theoretic background for establishing this connection, we observe, in analogy to \cite{ADo-Advances}, that it is natural to consider whether a given operator ideal is generic or strongly bounded. Therefore this work has two main objectives, studying the existence of a universal operator for complements of ideals and genericity of ideals. The current work, particularly the results in section 5, is a natural extension of the work on the universality of spaces in the language of operator ideals.

In section 2 we provide the basic definitions and, in several subsections, define the many ideals we are considering. Many of the classes of operators first appeared in a recent paper \cite{BCFrWa-JFA}, but rather than refer the reader to this paper for the definitions and properties, we have opted to include this information. For each subclass defined, we list what is known  regarding universal operators, space factorization, and genericity. Along the way several open problems are presented which we hope will motivate more research. In section 3, we provide examples of operator ideals whose complements have universal operators. In section 4 we give examples of operator ideals whose complements do not have universal operators. In section 5 we introduce the necessary descriptive set theoretic definitions and prove many of the the foundational results regarding the complexity of the the factorization relations. The main result in the section shows several natural ideals are indeed generic and strongly bounded. 

\section{Basic Definitions, Operator Ideals, Subclasses and Universal Operators}

Let $\textbf{Ord}$ be the class of ordinals and $\omega_1$ denote the first uncountable ordinal. For a set $\Lambda$, let $\Lambda^{<\mathbb{N}}$ denote the set of all finite sequences whose members lie in $\Lambda$,  including the empty sequence, denoted $\varnothing$. We refer to the members of $\Lambda^{<\nn}$ as \emph{nodes}. We order $\Lambda^{<\mathbb{N}}$ by extension. That is, we write  $s \preceq t$ if $s$ is an initial segment of $t$. Let $|s|$ denote the length of the node $s$ and for $i\leqslant |s|$, let $s|_i$ be the initial segment of $s$ having length $i$. A set $T \subset \Lambda^{<\mathbb{N}}$ is called a \emph{tree} on $\Lambda$ if it  downward closed with respect to the order $\preceq$. For $t \in T$, let $T(t)=\{s \in \Lambda^{<\mathbb{N}}:t\smallfrown s \in T\}$. Here, $t\smallfrown s$ denotes the concatenation of $t$ and $s$. Note that if $T$ is a tree, so is $T(t)$.  The set $MAX(T)$ consists of all $t \in T$ so that $T(t)=\{\emptyset\}$. 

Given a tree $T\subset \Lambda^{<\nn}$, let $T'=T \setminus MAX(T)$. Define the set of higher order derived tree $T^\xi$ for $\xi \in \textbf{Ord}$ as follows. We let $$T^0=T,$$  $$T^{\xi+1}=(T^\xi)',$$  $$T^\xi=\cap_{\zeta <\xi} T^\zeta, \text{\ \ \ }\xi\text{\ is a limit ordinal.}$$ We say $T$ is \emph{well-founded} if there exists an ordinal $\xi$ such that $T^\xi=\varnothing$, and in this case we let $o(T)$ be the minimum such ordinal. If $T$ is such that $T^\xi\neq \varnothing$ for all $\xi \in \textbf{Ord}$, then we say $T$ is \emph{ill-founded} and write $o(T)=\infty$. We note that $T\subset \Lambda^{<\nn}$  is ill-founded if and only if there exists a sequence $(t_i)_{i=1}^\infty \subset T$ such that $t_1\prec t_2\prec \ldots$ if and only if there exists a sequence $(x_i)_{i=1}^\infty\subset \Lambda$ such that $(x_i)_{i=1}^n\in T$ for all $n\in \nn$. 
 
Throughout,  $2^\nn$ will denote the Cantor set and $C(2^\nn)$ will denote the space of continuous functions on the Cantor set. Moreover, $\textbf{Ban}$ will denote the class of all Banach spaces and $\textbf{SB}$ will denote the set of all closed subsets of $C(2^\nn)$ which are linear subspaces.   We will let $\mathfrak{L}$ denote the class of all operators between Banach spaces and $\mathcal{L}$ the class of operators between separable Banach spaces.  Given $X,Y\in \textbf{Ban}$, we let $\mathfrak{L}(X,Y)$ denote all operators from $X$ into $Y$.     Given a subclass $\mathfrak{J}$ of $\mathfrak{L}$ and $X,Y\in \textbf{Ban}$,  we let $\mathfrak{J}(X,Y)=\mathfrak{J}\cap \mathfrak{L}(X,Y)$.  We recall that a subclass $\mathfrak{J}$ of $\mathfrak{L}$ is said to have the \emph{ideal property} if for any $W,X,Y,Z\in \textbf{Ban}$, any $C\in \mathfrak{L}(W,X)$, $B\in \mathfrak{J}(X,Y)$, and any $A\in \mathfrak{L}(Y,Z)$, $ABC\in \mathfrak{J}(W,Z)$.   We say a subclass $\mathfrak{J}$ of $\mathfrak{L}$ is an \emph{operator ideal} in the sense of Pietsch \cite{Pie-book} provided that 

\begin{enumerate}[(i)]\item $\mathfrak{J}$ has the ideal property, \item $I_\mathbb{R}\in\mathfrak{J}$, \item for every $X,Y\in \textbf{Ban}$, $\mathfrak{J}(X,Y)$ is a vector subspace of $\mathfrak{L}(X,Y)$. 
\end{enumerate}

If $\mathfrak{J}$ is an operator ideal, we denote the complement of $\mathfrak{J}$ by $\complement \mathfrak{J}$. We recall that $\mathfrak{J}$ is said to be \emph{injective} provided that whenever $X,Y,Z\in \textbf{Ban}$ and $j:Y\to Z$ is an isometric embedding, if $A:X\to Y$ is such that $jA\in \mathfrak{J}$, then $A\in \mathfrak{J}$. Let $\mbox{Space}(\mathfrak{J})$ be the $X \in \textbf{Ban}$ such that $I_X\in \mathfrak{J}$.  An ideal $\mathfrak{J}$ is called \emph{proper} if each $X \in \mbox{Space}(\mathfrak{J})$ is finite dimensional.  Let  $\mathfrak{J}^{super}$ denote the ideal of all $A:X  \to Y$ such that for any ultrafilter $\mathcal{U}$ on any set, the induced operator $A_\mathcal{U}:X_\mathcal{U} \to Y_\mathcal{U}$ is in $\mathfrak{J}$.

We will consider several subclasses of classical operator ideals. Suppose that  $\mathfrak{J}$ is a closed operator ideal and $(\mathfrak{J}_\xi)_{\xi<\omega_1}$ is a collection of subclasses of $\mathfrak{J}$ such that 
\begin{enumerate}
\item for each $\xi <\omega_1$, $\mathfrak{J}_\xi$ is closed, $\mathfrak{J}\setminus \mathfrak{J}_\xi$ is non-empty, and $\mathfrak{J}_\xi$ has the ideal property (Note: We do not require that each $\mathfrak{J}_\xi$ is an ideal),
\item  $\mathfrak{J}(X,Y) =\cup_{\xi <\omega_1} \mathfrak{J}_\xi(X,Y)$ for all separable Banach spaces $X$ and $Y$.
\end{enumerate}

\noindent In this case, we say $(\mathfrak{J}_\xi)_{\xi<\omega_1}$ \emph{separably refines} $\mathfrak{J}$.

If  $\mathfrak{J}$ is a closed operator ideal and $(\mathfrak{J}_\xi)_{\xi\in \textbf{Ord}}$ is a collection of subclasses of $\mathfrak{J}$ such that 
\begin{enumerate}
\item for each $\xi \in \textbf{Ord}$, $\mathfrak{J}_\xi$ is closed, $\mathfrak{J}\setminus \mathfrak{J}_\xi$ is non-empty, and $\mathfrak{J}_\xi$ has the ideal property,
\item $\mathfrak{J}(X,Y) =\cup_{\xi \in \textbf{Ord}} \mathfrak{J}_\xi(X,Y)$ for all  Banach spaces $X$ and $Y$.
\end{enumerate}
In this case, we say $(\mathfrak{J}_\xi)_{\xi\in \textbf{Ord}}$ \emph{refines} $\mathfrak{J}$.

Let $A:X \to Y$ and $B:Z \to W$ be operators. Then $A$ factors through $B$ if there are operators $C:X \to Z$ and $D:W \to Y$ so that $A=DBC$. Also we say that $A$ factors through a restriction of $B$ if there are subspaces $Z_1$ of $Z$ and $W_1$ of $W$ so that $B|_{Z_1}:Z_1 \to W_1$ and $A$ factors through $B|_{Z_1}$. Note that every operator $A:X \to Y$ between separable Banach spaces factors through a restriction of the identity on $C(2^\mathbb{N})$. For reference we isolate the following remark.

\begin{xrem}
Let $X,Z$ be Banach spaces. Then $I_X$ factors through a restriction of $I_Z$ if and only if $X$ is isomorphic to a subspace of $Z$, and $I_X$ factors through $I_Z$ if and only if $Z$ contains a complemented copy of $X$.  
\end{xrem}

Suppose $\mathfrak{J}\subset \mathfrak{L}$ is a class possessing the ideal property. Then $U:E \to F$ is \emph{universal} for $\complement \mathfrak{J}$ if $U$ factors through every element of $\complement \mathfrak{J}$. A subset $\mathcal{B}$ of $\complement \mathfrak{J}$ is \emph{universal} if for each $A \in \complement \mathfrak{J}$ there exists  $B \in \mathcal{B}$ so that $B$ factors through $A$. We say that \emph{there is no} $\complement \mathfrak{J}$-\emph{universal operator} if there is no $U$ in $\complement \mathfrak{J}$ which factors through every $A \in \complement \mathfrak{J}$. 


Consider the following easy, but useful, remark regarding universal operators.

\begin{xrem}\upshape If $\mathfrak{I},\mathfrak{J}$ are two classes of operators which possess the ideal property, then $\complement (\mathfrak{I}\cap \mathfrak{J})$ cannot have a universal factoring operator unless either $\mathfrak{I}\subset \mathfrak{J}$ or $\mathfrak{J}\subset \mathfrak{I}$.  Indeed, if $U$ is a universal factoring operator for $\complement(\mathfrak{I}\cap \mathfrak{J})$,  then we may assume without loss of generality that $U\in \complement\mathfrak{I}$. Then if $B\in \complement\mathfrak{J}$, $U$ factors through $B$, and $B\in \complement  \mathfrak{I}$. From this it follows that $\mathfrak{I}\subset \mathfrak{J}$.
\label{get away from me}
\end{xrem}

Let $\mathfrak{J}$ and $\mathfrak{M}$ be classes of operators with the ideal property. We say that $\mathfrak{J}$ has the $\mathfrak{M}$ \emph{space factorization property} if every $A \in \mathfrak{J}$ factors through a space $X \in \mbox{Space}(\mathfrak{M})$. When $\mathfrak{J}=\mathfrak{M}$ we say that $\mathfrak{J}$ has the space factorization property. Whether a given operator ideal has the space factorization property is an extensively studied topic \cite{DFJP-JFA,Hein-JFA}.
 Suppose $U\in \mathcal{L}$ is such that $A \in \mathfrak{J}\cap \mathcal{L}$ factors through a restriction of $U$. Then we say that $\mathfrak{J}$ is \emph{Bourgain-generic} (or just generic) if the identity on $C(2^\mathbb{N})$ (and therefore every operator $A \in \mathcal{L}$) also factors through $U$. Note that this is the an operator version of space the notion of Bourgain-generic defined in \cite{ADo-Advances}. We also note that Johnson and Szankowski \cite{JoSz-JFA} showed that there is no operator between separable Banach spaces through which every compact operator factors. Consequently, the weaker notion of factoring through a restriction is necessary to make this definition non-trivial for an operator ideal containing the compact operators.  

In what follows, for any infinite subset $M$ of $\nn$, $[M]$ will denote the infinite subsets of $M$. Let us recall the Schreier families from \cite{AlA-Dissertationes}. The Schreier families will be treated both as collections of finite subsets of $\nn$ as well as collections of empty or finite, strictly increasing sequences of natural numbers. The identification between subsets and strictly increasing sequences is the natural one obtained by listing the members of a set in strictly increasing order. We note that, since $[\nn]^{<\nn}$ is identified with a set of finite sequences of natural numbers, the downward closed subsets of $[\nn]^{<\nn}$ are trees with respect to the previously introduced initial segment ordering $\preceq$. Given a subset $\mathcal{G}$ of $[\nn]^{<\nn}$, we say $\mathcal{G}$ is \begin{enumerate}[(i)]\item \emph{hereditary} if $E\subset F\in \mathcal{G}$ implies $E\in \mathcal{G}$, \item \emph{spreading} if $(m_i)_{i=1}^k \in \mathcal{G}$ and $m_i\leqslant n_i$ for all $1\leqslant i\leqslant k$ imply $(n_i)_{i=1}^k\in \mathcal{G}$, \item \emph{compact} if it is compact in the identification $E\leftrightarrow 1_E\in \{0,1\}^\nn$, where $\{0,1\}^\nn$ has the product of the discrete topology, \item \emph{regular} if it is hereditary, spreading, and compact. 

\end{enumerate}

Given two subsets $E,F$ of $\nn$, we let $E<F$ denote the relationship that either $E=\varnothing$, $F=\varnothing$, or $\max E<\min F$. For $n\in \nn$, we let $n\leqslant E$ denote the relationship that either $E=\varnothing$ or $n\leqslant \min E$. 

We let $$\mathcal{S}_{0}=\{\varnothing\}\cup \{(n): n\in \nn\}.$$  Assuming that $\mathcal{S}_\xi$ has been defined for some ordinal $\xi<\omega_1$, let 
$$\mathcal{S}_{\xi+1}= \{\cup_{i=1}^n  E_i : n\leqslant E_1 < E_2 < \cdots <E_n, E_i \in \mathcal{S}_{\xi}\}\cup \{\emptyset\}.$$
If $\xi<\omega_1$,  is a limit ordinal, we fix a specific sequence $(\xi_n)_{n=1}^\infty$ with $\xi_n\uparrow \xi$ and define
$$\mathcal{S}_\xi = \{ E : \exists n \leqslant E \in \mathcal{S}_{\xi_n}\}\cup \{\emptyset\}.$$  Some of the results we cite regarding the Schreier families and related notions depend upon the specific sequence $(\xi_n)_{n=1}^\infty$, for which we refer the reader to \cite{AlA-Dissertationes}. For our purposes, it is sufficient to remark that for some specific choice of $(\xi_n)_{n=1}^\infty$, the stated results hold.

For an ordinal $0<\xi<\omega_1$, a bounded sequence $(x_n)$ in a Banach space $X$ is  an $\ell_1^\xi$ \emph{ spreading model} if there exists  $\delta>0$ such that for every scalar sequence $(a_i)$ and $F \in \mathcal{S}_\xi$,  
$$\delta \sum_{i\in F}|a_i| \leqslant \|\sum_{i\in F} a_i x_i\|.$$ For $0<\xi<\omega_1$, we say $(x_n)$ is a $c_0^\xi$ \emph{spreading model} if $\inf_i \|x_i\|>0$ and there exists a constant $\delta>0$ such that for every scalar sequence $(a_i)$ and $F\in \mathcal{S}_\xi$, $$\delta\|\sum_{i\in F} a_i x_i\|\leqslant \max_{i\in F}|a_i|.$$  We remark that, by Rosenthal's $\ell_1$ dichotomy and properties of $\ell_1^\xi$ spreading models, if $(x_n)$ is an $\ell_1^\xi$ spreading model, then either some subsequence of $(x_n)$ is equivalent to the $\ell_1$ basis, or there exist $p_1<q_1<p_2<q_2<\ldots$ such that $(x_{p_n}-x_{q_n})$ is a weakly null $\ell_1^\xi$ spreading model. Furthermore, it is easy to see that any $c_0^\xi$ spreading model is a weakly null sequence.

A sequence $(x_n)$ is $\mathcal{S}_\xi$-\emph{unconditional} if there exists $\delta>0$ such that for every scalar sequence $(a_i)$ and $F \in \mathcal{S}_\xi$,  
$$\delta\|\sum_{i\in F} a_i x_i\| \leqslant \|\sum_{i=1} ^\infty a_i x_i\|.$$

We recall the repeated averages hierarchy from \cite{AMerTs-Israel}: For each countable ordinal $\xi$ and $L \in [\mathbb{N}]$,  the sequence $(\xi_n^L)_{n=1}^\infty$ is a convex block sequence of $(e_i)$ such that $L=\cup_{n=1}^\infty \mbox{supp }\xi_n^L $ and $\mbox{supp }\xi^L_n \in MAX(\mathcal{S}_{\xi})$. A sequence $(x_n)$ in a Banach space $X$ is $(\xi,L)$-convergent to a vector $x \in X$ if 
$$\|\sum_{i}\xi^L_n(i) x_i - x\| \to 0,~n \to \infty$$
Note that if $(x_n)$ is $(\xi,L)$ convergent to $0$ then it is weakly null.
Recall the following deep theorem (see \cite{AGas-TAMS,AGodR-book,AMerTs-Israel}).

\begin{theorem}
Let $(x_n)$ be a weakly null sequence in a Banach space $X$ and $\xi$ be a countable ordinal. Then one of the following mutually exclusive properties holds:
\begin{enumerate}
    \item For each $M \in [\mathbb{N}]$ there exists and $L \in [M]$ so that $(x_n)$ is $(\xi,L)$ convergent to $0$.
    \item There exists an $M\in [\mathbb{N}]$ so $(x_i)_{i \in M}$ is  an $\ell_1^\xi$ spreading model and is $\mathcal{S}_\xi$-unconditional.
\end{enumerate}
\label{the dichotomy}
\end{theorem}

For each countable ordinal $\xi$, the Schreier space $X_\xi$ is the completion of $c_{00}$ with respect to the norm
$$\|x\|_{X_{\xi}}= \sup_{F \in \mathcal{S}_\xi} \sum_{i \in F}|x(i)|.$$
Here $x=(x(1),x(2), \ldots ) \in c_{00}$. 

For $\xi<\omega_1$, let $T_\xi$ denote the Tsirelson space of order $\xi$. Using the notation from \cite{ATo-book,ATol-Memoirs}, $T_\xi = T[\frac{1}{2},\mathcal{S}_\xi]$. Here $\xi=1$ is the usual Figiel-Johnson Tsirelson space \cite{FiJo-Compositio}. Rather than recall the definition, we note that $T_\xi$ is a reflexive space such that the canonical $c_{00}$ basis is an unconditional basis for $T_\xi$ and such that every seminormalized block sequence with respect to the canonical basis is an $\ell_1^\xi$ spreading model.

In the following subsections we define several operator ideals and their subclasses. Note that all of the subclasses defined below satisfy the ideal property and some have been shown to be operator ideals themselves. In the interest of providing the necessary background, we have organized what is known about each of these classes as follows.
\begin{itemize}
    \item[(U)] We state what is known regarding universal operators for the complement of the class.
    \item[(F)] We state whether the class has the space factorization property or if it is an open question.
    \item[(G)] We state whether the class is (Bourgain) generic or if it is a open question.
\end{itemize}

In a slight departure from previous papers by the current authors on this subject, we will adopt the notation for operators ideals that is used in the excellent survey by J. Diestel, H. Jarchow and A. Pietsch \cite{DJP-Handbook}.

\subsection{Approximable and Compact operators} Let  $\overline{\mathfrak{F}}$ and $\mathfrak{K}$ denote the closed operator ideals of compact and closure of finite rank operators, respectively. 

\begin{itemize}
    \item[(U)] It is an open problem as to whether there is a universal operator for $\complement \overline{\mathfrak{F}}$.  W.B. Johnson showed \cite{Jo-Colloquium} that the formal identity from $\ell_1$ to $\ell_\infty$ is $\complement \mathfrak{K}$-universal.  
    \item[(F)] It is easy to see that neither $\overline{\mathfrak{F}}$ nor $\mathfrak{K}$ has the space factorization property. 
    \item[(G)] The operator ideals  $\overline{\mathfrak{F}}$ and $\mathfrak{K}$ are not generic. T. Figiel \cite{Fi-Studia} constructed a reflexive space $Z$ such that every element of $\overline{\mathfrak{F}}$ factors through $Z$ and every element of $\mathfrak{K}$ factors through a subspace of $Z$.   
\end{itemize}

\subsection{Completely Continuous Operators and the Dunford-Pettis Operator Ideal}

An operator $A:X \to Y$ is \emph{completely continuous} provided that for every weakly null sequence $(x_n)_{n=1}^\infty \subset B_X$, $(Ax_n)_{n=1}^\infty$ is norm null. These operators form a closed ideal which we will denote $\mathfrak{V}$. An operator $A:X\to Y$ is in the ideal of Dunford Pettis $\mathfrak{DP}$
if $\lim_n y_n^*(Ax_n)=0$ for every weakly null sequence $(y_n^*)$ in $Y^*$ and weakly null sequence $(x_n)$ in $X$. We note that in some literature completely continuous operators are called Dunford Pettis operators. In this paper we adopt the terminology of \cite{DJP-Handbook} and consider the ideal $\mathfrak{DP}$ to be the operators so that $\mbox{Space}(\mathfrak{DP})$ coincides the spaces having the Dunford-Pettis property.

\begin{itemize}
\item[(U)] M. Girardi and W.B. Johnson \cite{GirJo-Israel} showed that there is no $\complement \mathfrak{V}$-universal operator. In Theorem \ref{come at me bro} we give an alternative proof and strengthening of this theorem that uses a complexity argument from descriptive set theory. Also in Theorem \ref{come at me bro}, we show that there is no  $\complement \mathfrak{DP}$-universal operator.
\item[(F)] The ideal $\mathfrak{V}$ does not have the space factorization property \cite[page 60]{Pie-book}. According to \cite{DJP-Handbook} it is an open question as to whether $\mathfrak{DP}$ has the space factorization property.
\item[(G)] In Theorem \ref{lots of stuff}, we show that $\mathfrak{V}$ and $\mathfrak{DP}$ are generic.
\end{itemize}

\subsection{Weakly Compact Operators and Subclasses} An operator $A:X\to Y$ is \emph{weakly compact} if $A(B_X)$ ($B_X$ is the unit ball of $X$) is relatively weakly  compact. These operators form a closed ideal that we denote $\mathfrak{W}$.  
\begin{itemize}
    \item[(U)] J. Lindenstrauss and A. Pe\l cyz\'{n}ski \cite{LindPel-Studia} showed that the operator $s:\ell_1\to \ell_\infty$ given by $s((a_i)_{i=1}^\infty)=(\sum_{i}^n a_i)_{n=1}^\infty$ is  $\complement \mathfrak{W}$-universal.
    \item[(F)] The famous theorem of Davis, Figiel, Johnson and Pe\l czy\'{n}ski \cite{DFJP-JFA} states that $\mathfrak{W}$ has the space factorization property.
    \item[(G)]  In Theorem \ref{lots of stuff} we prove that $\mathfrak{W}$ is generic.
\end{itemize}
For $\mathfrak{W}^{super}$ we have the following:
\begin{itemize}
    \item[(U)] In Theorem \ref{thats just super} we observe that $\complement \mathfrak{W}^{super}$ has a universal operator.
    \item[(F)] It is well-known and easy to see that $\mathfrak{W}^{super}$ does not have the space factorization property. An example of this is given in  \cite[Proposition 1.2]{BC-Scand}.
    \item[(G)]  In Corollary \ref{not generic}  we show that $\mathfrak{W}^{super}$ not generic.
\end{itemize}

Fix a countable ordinal $\xi$. An operator $A:X \to Y$ is called $\mathcal{S}_\xi$-\emph{weakly compact} if either $\dim X<\infty$ or  for each normalized basic sequence $(x_n)$ in $X$ and $\varepsilon >0$ there are  $(n_i)_{i=1}^\ell \in \mathcal{S}_\xi$ and scalars $(a_i)_{i=1}^\ell$ so that 
$$\|\sum_{i=1}^\ell a_i A x_{n_i}\| <\varepsilon \max_{1\leqslant \ell \leqslant k}|\sum_{i=\ell}^k a_i|.$$ 
Denote this subclass of $\mathfrak{W}$ by $\mathcal{S}_\xi$-$\mathfrak{W}$. It is easily seen that $\mathcal{S}_0$-$\mathfrak{W}$ is simply the class of compact operators.  In \cite{BFr-Extracta}, it is shown that $(\mathcal{S}_\xi$-$\mathfrak{W})_{\xi<\omega_1}$ separably refines $\mathfrak{W}$. Moreover for each $\xi<\omega_1$, the class $\mathcal{S}_\xi$-$\mathfrak{W}$ forms a closed ideal. In \cite{BC-Scand}, the authors show that these ideals coincide with two natural, but ostensibly different, subclasses. One of these classes is the class of $\xi$-Banach-Saks operators, defined in a subsequent subsection. 

\begin{itemize}
    \item[(U)] In Theorem \ref{soo bad} we prove that for each $0<\xi<\omega_1$, there is no $\complement \mathcal{S}_\xi$-$\mathfrak{W}$-universal operator. 
    \item[(F)] For each $0<\xi<\omega_1$,  $\mathcal{S}_\xi$-$\mathfrak{W}$ has the space factorization property \cite{BC-Scand}.
    \item[(G)]  In Theorem \ref{lots of stuff},  we show that  $\mathcal{S}_\xi$-$\mathfrak{W}$ is generic for each $0<\xi<\omega_1$. 
\end{itemize}

\subsection{Strictly Singular Operators and subclasses} An operator $A:X\to Y$ is \emph{strictly singular} provided that for any infinite dimensional subspace $Z$ of $X$, $A|_Z$ is not an isomorphic embedding. These operators form a closed ideal that will be denoted by  $\mathfrak{S}$. This operator ideal is proper. 

\begin{itemize}
    \item[(U)]  T. Oikhberg showed in \cite{Oik-Positivity} that there is no $\complement \mathfrak{S}$-universal operator.
     \item[(F)] The class $\mathfrak{S}$ does not have the space factorization property.
    \item[(G)] In Theorem \ref{lots of stuff} we show that $\mathfrak{S}$ is generic.
\end{itemize}
An operator $A:X\to Y$ is \emph{finitely strictly singular} if for any $\varepsilon > 0$, there exists $n \in \mathbb{N}$ so that for any
$E \subset X$ with $\mbox{dim }E = n$, there exists $x \in E$ with $\|Ax\| < \varepsilon\|x\|$.  
\begin{itemize}
    \item[(U)] Oikhberg showed in \cite{Oik-Positivity} there is a $\complement \mathfrak{FS}$-universal operator (see Proposition \ref{thats just super}).
    \item[(F)] The class $\mathfrak{FS}$ does not have the space factorization property.
    \item[(G)] Whether $\mathfrak{FS}$ is a generic class remains an open question. We conjecture that it is not.
\end{itemize}

Fix  $\xi<\omega_1$. An operator $A:X \to Y$ is called $\mathcal{S}_\xi$-strictly singular if $\dim X<\infty$ or for each normalized basic sequence $(x_n)$ in $X$ and $\varepsilon >0$,  there are $(n_i)_{i=1}^\ell \in \mathcal{S}_\xi$ and scalars $(a_i)_{i=1}^\ell$ such that 
$$\|\sum_{i=1}^\ell a_i A x_{n_i}\| <\varepsilon \|\sum_{i=1}^\ell a_i x_{n_i}\|.$$ 
Denote the subclass of $\mathcal{S}_\xi$-strictly singular operators by $\mathcal{S}_\xi$-$\mathfrak{S}$.  It is clear that $\mathcal{S}_0$-$\mathfrak{S}$ coincides with the class of compact operators. These classes where defined in \cite{AnDoSiTr-Israel} and have been extensively studied \cite{AnB-QM,B-Israel,BDo-Mathematika,Popov-Houston}. In particular, while it is shown that $(\mathcal{S}_\xi$-$\mathfrak{S})_{\xi< \omega_1}$ separably refines $\mathfrak{S}$. In general, $\mathcal{S}_\xi $-$\mathfrak{S}$ need not satisfy the additive property of being an operator ideal \cite{OTe}. In the following, let $0<\xi<\omega_1$.

\begin{itemize}
    \item[(U)] In Theorem \ref{soo bad} we prove that  $\complement\mathcal{S}_\xi$-$ \mathfrak{S}$ does not have a universal operator.
    \item[(F)] The class $\mathcal{S}_\xi $-$\mathfrak{S}$ does not have the space factorization property.
    \item[(G)] In Theorem \ref{lots of stuff}, we show that $\mathcal{S}_\xi$-$\mathfrak{S}$ is generic.
\end{itemize}

Fix an operator $A:X\to Y$ and $K \geqslant 1$. Let 
$$T_\mathfrak{S}(A,X,Y,K)=\{(x_i)_{i=1}^n \in S_X^{<\mathbb{N}} : \mbox{$(x_i)_{i=1}^n$ is $K$-basic and $(Ax_i)_{i=1}^n$ $K$-dominates $(x_i)_{i=1}^n$}\}.$$
Let $r_\mathfrak{S}(X,Y,A)= \sup_{K\geqslant 1}o(T_\mathfrak{S}(A,X,Y,K))$ and for each ordinal $\xi$, let 
$$\mathfrak{S}^\xi (X,Y)=\{A \in \mathfrak{L}(X,Y): r_\mathfrak{S}(X,Y,A) \leqslant \omega^\xi\}.$$

These operators were defined in \cite{BCFrWa-JFA}, where it is proved that $(\mathfrak{S}^\xi)_{\xi \in \textbf{Ord}}$ refines $\mathfrak{S}$. Moreover, $\mathfrak{S}^1=\mathfrak{FS}$ and $\mathcal{S}_1$-$ \mathfrak{S} \not\subset \mathfrak{S}^\xi$ for all $0<\xi \in \textbf{Ord}$. 

\begin{itemize}
    \item[(U)] In Theorem \ref{soo bad}, we will show that the aforementioned Oikhberg result for $\mathfrak{FS}$ cannot be extended to the classes $\mathfrak{S}^\xi$. In particular, for each $1<\xi$, there is no $\complement \mathfrak{S}^\xi$-universal operator.
    \item[(F)] The class $\mathfrak{S}^\xi$ does not have the space factorization property.
    \item[(G)] We do not know if $\mathfrak{S}^\xi$ is generic for any $0<\xi<\omega_1$. 
\end{itemize}

\subsection{$E$-singular operators} Let $E$ be a Banach space. An operator $A:X \to Y$ is called \emph{$E$-singular} provide there is no subspace $X'$ of $X$ isomorphic to $E$ so that $A|_{X'}$ is an isomorphic embedding. Denote this collection of operators $\mathfrak{S}E$. It is observed in \cite[Section 8]{BCFrWa-JFA} that $\mathfrak{S}E$ is a closed operator ideal if and only if $\mathfrak{S}E$ is closed under addition. In particular,  $\mathfrak{S}E$ always satisfies the ideal property.

In the current paper, we mostly consider the particular cases $\ell_1$-singular,  $\mathfrak{Sl}_1$,  and $c_0$-singular $\mathfrak{Sc}_0$. The members of the class $\mathfrak{Sl}_1$ are called the \emph{Rosenthal operators}.

\begin{itemize}
    \item[(U)] It was observed in \cite{DJP-Handbook} that for any non-empty set $\Gamma$,  any isomorphic embedding from $E \to \ell_\infty(\Gamma)$ is $\complement \mathfrak{S}E$-universal. 
    \item[(F)] The ideal $\mathfrak{Sl}_1$ has the space factorization property \cite{Hein-JFA}. Ghoussoub and Johnson \cite{GhJo-PAMS} showed that $\mathfrak{Sc}_0$ does not have the space factorization property. It is an open question as to whether $\mathfrak{Sl}_p$ has the space factorization property for any $1<p<\infty$. 
    \item[(G)] In Theorem \ref{lots of stuff}, we show that $\mathfrak{S}$ is generic, and therefore $\mathfrak{S}E$ is generic for any infinite dimensional Banach space $E$. 
    
\end{itemize}

Fix a countable ordinal $0<\xi<\omega_1$ and suppose that the basis of $(e_i)$ of $E$ is isometric to all of it's subsequences (i.e. $1$-spreading).  An operator $A:X \to Y$ is called $\mathcal{S}_\xi$-$E$ \emph{singular} if for each normalized basic sequence $(x_n)$ in $X$ and $\varepsilon >0$,  there are $(n_i)_{i=1}^\ell \in \mathcal{S}_\xi$ and scalars $(a_i)_{i=1}^\ell$ such that either 
$$\|\sum_{i=1}^\ell a_i A x_{n_i}\| <\varepsilon \|\sum_{i=1}^\ell a_i e_{i}\|$$ or $$\|\sum_{i=1}^\ell a_i e_i\| < \ee \|\sum_{i=1}^\ell a_i x_{n_i}\|.$$
Denote the subclass of $\mathcal{S}_\xi$-$E$-singular operators by $\mathcal{S}_\xi$-$\mathfrak{S}E$. Note that $\mathcal{S}_\xi$-$\mathfrak{S}E$-singular depends on the basis chosen for $E$ and can change for different bases of the same space (e.g. consider the summing basis and standard basis of $c_0$). These classes were defined in \cite{BFr-Extracta}. There it was shown that $(\mathcal{S}_\xi$-$\mathfrak{S}E)_{\xi < \omega_1}$ separably refines $\mathfrak{S}E$.   

\begin{itemize}
    \item[(U)] In Theorem \ref{man thats good}, we will show that for each countable ordinal $\xi$, there are $\complement \mathcal{S}_\xi$-$\mathfrak{Sc}_0$ and $\complement \mathcal{S}_\xi$-$ \mathfrak{Sl}_1$ universal operators. Here we are considering the canonical unit vector basis for both $\ell_1$ and $c_0$. Whether the complements of the classes $\mathcal{S}_\xi$-$ \mathfrak{Sl}_p$ have universal operators for any $1 < p < \infty$ is an open problem.
    \item[(F)] In \cite{BC-Scand}, it is shown that $\mathcal{S}_\xi$-$ \mathfrak{Sl}_1$ has the space factorization property. The other cases remain open.
    \item[(G)] In Theorem \ref{lots of stuff}, we show that $\mathfrak{S}_1$-$\mathfrak{S}$ is generic, and therefore  $\mathcal{S}_\xi$-$ \mathfrak{S}E$ is generic for each $0<\xi<\omega_1$ and $E$ with a $1$-spreading basis. 
\end{itemize}

Now assume $E$ has a normalized Schauder basis $(e_i)_{i=1}^\infty$. Fix an operator $A:X\to Y$ and $K \geqslant 1$. Let

\begin{equation}
    \begin{split}
  T_{\mathfrak{S}E}(A,X,Y,K)=\{(x_i)_{i=1}^n \in S_X^{<\mathbb{N}} : & (Ax_i)_{i=1}^n \text{ is $K$ basic, $K$-dominates } (e_i)_{i=1}^n,\\
  & (e_i)_{i=1}^n \mbox{ $1$-dominates } (x_i)_{i=1}^n \}.      
    \end{split}
\end{equation}

Let $r_E(X,Y,A)= \sup_{K\geqslant 1}o(T_{\mathfrak{S}E}(A,X,Y,K))$ and, for each ordinal $\xi$, let 
$$\mathfrak{S}E^\xi(X,Y)=\{A \in \mathfrak{L}(X,Y): r_{\mathfrak{S}E}(X,Y,A) \leqslant \omega^\xi\}.$$ In \cite{BCFrWa-JFA}, it is proved that $(\mathfrak{S}E^\xi)_{\xi \in \textbf{Ord}}$ refines $\mathfrak{S}E$.

\begin{itemize}
    \item[(U)] In Theorem \ref{man thats good}, we show that  $\complement\mathfrak{Sc}_0^\xi$ has a universal operator for each $\omega\leqslant \xi \in \textbf{Ord}$. It was shown by Oihkberg in \cite{Oik-Positivity} that $\complement \mathfrak{Sl}_p^1$ and $\complement \mathfrak{Sc}_0^1$ have universal operators, and in fact the given universal operator for $\mathfrak{FS}$ coincides with the given universal operator for $\mathfrak{Sl}_2^1$.  For all $1<\xi$ and $1\leqslant p<\infty$ (resp. for all $1<\xi<\omega$), it is an open question   as to whether  $\complement\mathfrak{Sl}_p^\xi$ (resp. $\complement \mathfrak{Sc}_0^\xi$) has a universal operator. Once again we let $(e_i)$ be the unit vector basis for $\ell_p$ and $c_0$ in the definition of $\complement\mathfrak{Sl}_p$ and $\complement\mathfrak{Sc}_0$.
    \item[(F)] In \cite{BC-Scand}, the authors show that $\mathfrak{Sl}_1^{\omega^\xi}$ has the $\mathfrak{Sl}_1^{\omega^\xi+1}$ space factorization property. For every ordinal $0<\xi$, whether $\mathfrak{Sl}_p^\xi$ has the $\mathfrak{Sl}_p$ factorization property is open. The positive result is due to the projectivity of $\ell_1$, while the result of Ghoussoub and Johnson shows that there exists an ordinal $\xi$ such that $\mathfrak{Sc}_0^\xi$ does not have the $\mathfrak{Sc}_0$ factorization property. In light of these facts, we conjecture that for $1<p<\infty$, there exists an ordinal $\xi$ such that $\mathfrak{Sl}_p^\xi$ does not have the $\mathfrak{Sl}_p$ factorization property. 
    \item[(G)] We do not know if $\mathfrak{S}E^\xi$ is generic for any $\xi<\omega_1$ or Banach space $E$. We conjecture that these classes are never generic.
\end{itemize}

\subsection{Asplund Operators} Let $\mathfrak{D}$ denote the ideal of \emph{Asplund} (or \emph{decomposing}) operators.   We recall that $A:X\to Y$ is said to be Asplund if whenever $Z$ is a separable subspace of $X$, $(A|_Z)^*Y^*$ is a separable subspace of $Z^*$.   

\begin{itemize}
    \item[(U)] Stegall \cite{Stegall-TAMS} showed that, if $(f_n)_{n=1}^\infty\subset L_\infty[0,1]$ is an enumeration of the collection of indicator functions of the dyadic intervals $\{[\frac{j-1}{2^n}, \frac{j}{2^n}): n\in \nn\cup \{0\}, 1\leqslant j\leqslant 2^n\}$, and if $(e_n)_{n=1}^\infty$ is the canonical $\ell_1$ basis, the map $e_n\mapsto f_n$ is universal for the complement of the class of Asplund operators. 
    \item[(F)] In \cite{DFJP-JFA} it is shown that $\mathfrak{D}$ has the space factorization property.
    \item[(G)] In Theorem \ref{lots of stuff} we observe that $\mathfrak{D}$ is generic.
\end{itemize}

For each $\xi\in \textbf{Ord}$, let $\mathfrak{D}_\xi$ denote the class of operators with Szlenk index not exceeding $\omega^\xi$ (see \cite{Br-JOT} for definition). Brooker showed that $(\mathfrak{D}_\xi)_{\xi \in \textbf{Ord}}$ refines $\mathfrak{D}$, $(\mathfrak{D}_\xi)_{\xi<\omega_1}$ separably refines $\mathfrak{D}$,   each $\mathfrak{D}_\xi$ is a closed operator ideal, and $\mathfrak{D}_0$ coincides with the class of compact operators.  

\begin{itemize}
    \item[(U)] Brooker \cite{Brooker-PreprintAbsolute} showed that for each $\xi<\omega_1$,  there is an operator in $\mathfrak{D}_{\xi+1}$ that is $\complement \mathfrak{D}_\xi$-universal.
    \item[(F)] Brooker \cite{Brooker-PreprintAbsolute} showed that for every $\xi\in \textbf{Ord}$, $ \mathfrak{D}_\xi$ has the $\mathfrak{D}_{\xi+1}$ space factorization property.
    \item[(G)] In Theorem \ref{strongly bounded}, we note that $\mathfrak{D}_\xi$ is not generic for any $\xi<\omega_1$.
\end{itemize}

\subsection{Banach-Saks and weak Banach-Saks operators}  An operator $A:X \to Y$ is \emph{Banach-Saks} if for every bounded sequence $(x_n)\subset X$, $(Ax_n)$ has a Cesaro summable subsequence. Let $\mathfrak{BS}$ denote the ideal of Banach-Saks operators. 

\begin{itemize}
    \item[(U)] In Theorem \ref{soo bad}, we show there is no $\complement\mathfrak{BS}$-universal operator.
    \item[(F)] The class $\mathfrak{BS}$ has the space factorization property \cite{Beauzamy-Seminar}.
    \item[(G)]  In Theorem \ref{lots of stuff}, we prove that $\mathfrak{BS}$ is generic.
\end{itemize}

We now recall the definition of $\xi$-Banach-Saks operators for $0<\xi< \omega_1$ \cite{BC-Scand}. We note that these are not subclasses of $\mathfrak{BS}$,  but rather increasing super-classes of $\mathfrak{BS}$ such that $\mathfrak{BS}=\mathfrak{BS}^1$ and $$\mathfrak{W}\cap \mathfrak{X} \subsetneq \bigcup_{0<\xi<\omega_1} \mathfrak{BS}^\xi \subsetneq \mathfrak{W}.$$ Here, $\mathfrak{X}$ denotes the closed ideal of separable range operators. 

Fix an operator $A:X\to Y$ and suppose that $(x_n)$ is a bounded sequence in $X$.  Define $BS((x_n),A)$ to be the smallest countable ordinal $\xi$ (if any such exists) such that $(Ax_n)$ is $\xi$-convergent. Combining Theorem \ref{the dichotomy} with standard characterizations of weakly null sequences, such an ordinal exists  if and only if $(Ax_n)$ has a weakly convergent subsequence. If no such ordinal $\xi$ exists,   then we set $BS((x_n), A)=\omega_1$.  

 For $0<\xi<\omega_1$, we  say $A:X\to Y$ is $\xi$-\emph{Banach-Saks} provided that for every bounded sequence $(x_n)$ in $X$, $BS((x_n), A)\leqslant \xi$. Let  $\mathfrak{BS}^\xi$ denote the class of $\xi$-Banach-Saks operators. Note that $\mathfrak{BS}^1=\mathfrak{BS}$.

For each $0<\xi<\omega_1$, we know from \cite{BCFrWa-JFA} that $\mathcal{S}_\xi$-$\mathfrak{Sl}_1\cap \mathfrak{W}=\mathcal{S}_\xi$-$\mathfrak{W}=\mathfrak{BS}^\xi$.
This equality means we can appeal to previous results to to see the following

\begin{itemize}
    \item[(U)] In Theorem \ref{soo bad} we show there is no $\complement \mathfrak{BS}^\xi$-universal operator.
    \item[(F)] In \cite{BC-Scand} the authors show that $\mathfrak{BS}^\xi$ has the space factorization property.
    \item[(G)]  In Theorem \ref{lots of stuff} we show that for each $0<\xi<\omega_1$, $\mathfrak{BS}^\xi$ is generic.
\end{itemize}

 An operator $A:X \to Y$ is weak-Banach-Saks if for every weakly null sequence $(x_n)$ in $X$, $(Ax_n)$ has a Cesaro summable subsequence. Let $\mathfrak{wBS}$ denote the ideal of weak-Banach-Saks operators.
 For each $0<\xi<\omega_1$,  the ideal of $\mathfrak{wBS}^\xi$ is defined  analogously to the class $\mathfrak{BS}^\xi$ by replacing the condition ``every bounded sequence" with ``every weakly null sequence.'' 
 
 \begin{itemize}
    \item[(U)] In Theorem \ref{soo bad}, we show that for each $0<\xi<\omega_1$,  $\complement \mathfrak{wBS}^\xi$ does not have a  universal operator.
    \item[(F)]  It is an open question as to whether any of these classes have the factorization property. 
    \item[(G)]  In Theorem \ref{lots of stuff},  for each $0<\xi<\omega_1$, we show that  $\mathfrak{wBS}^\xi$ is generic.
\end{itemize}

\section{Universal Operators: Positive results}

We begin this section by observing that the complements of many of the super ideals mentioned above have universal operators. The argument are a straightforward application of the following well-known ``gluing'' technique (see \cite{Baudier-Archiv} and \cite{Oik-Positivity} for examples of this technique).

\begin{proposition} Suppose that $\{U_n:E_n\to \ell_\infty: n\in \nn\}$ is a sequence of  operators such that $\dim E_n<\infty$ for each $n\in \nn$. If $B:X\to Y$ is such that there exists $c>0$ such that for each $n \in \mathbb{N}$, $U_n$ $c$-factors through $B|_Z$ for every finite codimensional subspace $Z$ of $X$, then the operator $$\oplus_{n=1}^\infty U_n:(\oplus_{n=1}^\infty E_n)_{\ell_1}\to \ell_\infty$$ factors through $B$. 
\label{factfact}
\end{proposition}

\begin{proof} Let $W_n=\ell_\infty$ for each $n\in \nn$. We recursively define spaces $F_n, Z_n\leqslant X$, $G_n\leqslant Y$ and operators $$C_n:E_n\to C(E_n)=F_n\leqslant Z_n,$$ $$A_n:G_n=B(F_n)\to W_n$$ such that $\|C_n\|, \|A_n\|\leqslant c$ for all $n\in \nn$, and such that $(G_n)_{n=1}^\infty$ is a Schauder FDD for its closed span in $Y$. First fix a sequence $(\ee_n)_{n=1}^\infty$ of numbers such that $\prod_{n=1}^\infty (1+\ee_n)<2$.  We next let $Z_1=X$ and choose any factorization $C_1:E_1\to X$, $A_1:Y\to W_1$ with $\|C_1\|, \|A_1\|\leqslant c$. Let $F_1=C_1(E_1)$ and  $G_1=BC_1(E_1)$. Now assuming $G_1, \ldots, G_n$ have been defined, let $G=\text{span}\cup_{i=1}^n G_i$. Fix  a finite subset $S_{n+1}$ of $B_{Y^*}$ such that for each $y\in G$, $$(1+\ee_n)\max_{y^*\in S_{n+1}} |y^*(y)| \geqslant \|y\|.$$  Let $Z_{n+1}= \cap_{y^*\in S_{n+1}} \ker(B^*y^*)$. Now fix a factorization $C_{n+1}:E_{n+1}\to Z_{n+1}$, $A_{n+1}:Y\to W_{n+1}$ with $\|C_{n+1}\|, \|A_{n+1}\|\leqslant c$. Let $F_{n+1}=C_{n+1}(E_{n+1})$ and $G_{n+1}=B(F_{n+1})$. From the construction, it follows that if $y\in G$ and $y'\in G_{n+1}$, then $y'=Bx'$ for some $x'\in Z_{n+1}$, whence for any $y^*\in S_{n+1}$, $$y^*(y')= y^*(Bx')= B^*y^*(x')=0.$$  Therefore $$\|y+y'\| \geqslant \max_{y^*\in S_{n+1}} |y^*(y+y')|= \max_{y^*\in S_{n+1}} |y^*(y)| \geqslant (1+\ee_n)^{-1}\|y\|.$$  This completes the recursive construction.

Applying the last inequality from the previous paragraph iteratively, for any $1\leqslant m<n$ and any $y_i\in G_i$, $1\leqslant i\leqslant n+1$, $$\|\sum_{i=1}^{n+1} y_i\|\leqslant (1+\ee_n)\|\sum_{i=1}^n y_i\|\leqslant \ldots \leqslant \Bigl[\prod_{i=m}^n (1+\ee_i)\Bigr]\|\sum_{i=1}^m y_i\|\leqslant 2\|\sum_{i=1}^m y_i\|.$$  From this it follows that $(G_i)_{i=1}^\infty$ is a Schauder FDD for its closed span, call it $Y_1$. From this it follows that the operator $\tilde{A}:Y_1\to (\oplus_{n=1}^\infty W_n)_{\ell_\infty}$ given by $\tilde{A}|_{G_n}= A_n$ is well-defined and bounded. Let $A:Y\to \ell_\infty=(\oplus_{n=1}^\infty W_n)_{\ell_\infty}$  be any extension of $\tilde{A}$.  The operator $C:(\oplus_{n=1}^\infty E_n)_{\ell_1}\to X$ given by $C|_{E_n}=C_n$ is bounded, and $ABC=U$. 
\end{proof}

It is well-known that an operator is in $\complement(\mathfrak{W}^{super})$ if and only if the  finite summing operators $s_n:\ell_1^n\to \ell_\infty^n$, $n=1, 2, \ldots$, factor uniformly through it. Now fix $1\leqslant p\leqslant \infty$. It is easy to see that, if for each $n\in \nn$, $j_n:\ell_p^n\to \ell_\infty$ is an isometric embedding, then an operator is in $\complement (\mathfrak{Sl}_p^{super})$ (resp. $\complement (\mathfrak{Sc}_0^{super})$ if $p=\infty$) if and only if the operators $j_n:\ell_p^n\to \ell_\infty$ uniformly factor through it. Moreover, it is easy to see that an operator $B:X\to Y$  lies in $\complement (\mathfrak{W}^{super})$ (resp. $\complement (\mathfrak{Sl}_p^{super})$) if and only if there exists a constant $c$ such that the operators $s_n:\ell_1^n\to \ell_\infty^n$ (resp. $j_n:\ell_p^n\to \ell_\infty$) $c$ factors through $B|_Z$ for every $Z\leqslant X$ such that $\dim X/Z<\infty$ (see \cite{CDil-Studia} for the first statement and \cite{BCFrWa-JFA} for the second).  Combining these facts with Proposition \ref{factfact}, we obtain the following.

\begin{proposition}
The complements of the ideals $\mathfrak{W}^{super}$ and $\mathfrak{Sl}_p^{super}$ each have universal operators. In particular we have the following:
\begin{enumerate}
\item $B:X\to Y$ lies in $\complement\mathfrak{W}^{super}$ if and only if the operator $$\oplus_{n=1}^\infty s_n:(\oplus_{n=1}^\infty \ell_1^n)_{\ell_1}\to (\oplus_{n=1}^\infty \ell_\infty^n)_{\ell_\infty}$$
where  $s_n:\ell_1^n\to \ell_\infty^n$ are the finite summing operators given by $s_n(e_i)=\sum_{j=1}^i e_j$ factors through $B$. 
\item $B:X\to Y$ lies in $\complement \mathfrak{Sl}_p^{super}$ if and only if the operator  $$\oplus_{n=1}^\infty j_n: (\oplus_{n=1}^\infty \ell_p^n)_{\ell_1}\to (\oplus_{n=1}^\infty \ell_\infty)_{\ell_\infty}$$ factors throuh $B$. Here,  $j_n:\ell_p^n\to \ell_\infty$ is an isometric embedding. 
\end{enumerate}
\label{thats just super}
\end{proposition}

The next theorem contains the rest of our positive factoring results.

\begin{theorem} Fix $0<\xi<\omega_1$.  \begin{enumerate}[(i)]
\item The complement of $\mathcal{S}_\xi$-$\mathfrak{Sl}_1$ admits a universal factoring operator.  
\item The complement of $\mathcal{S}_\xi$-$\mathfrak{Sc}_0$ admits a universal factoring operator. 
\item If $\xi\geqslant \omega$, the complement of $\mathfrak{Sc}_0^\xi$ admits a universal factoring operator.  
\end{enumerate}
\label{man thats good}
\end{theorem}

\begin{proof}$(i)$ Recall that $X_\xi$ denotes Schreier space of order $\xi$.  Let $i_\xi:\ell_1\to X_\xi$ be the formal identity and let $j:X_\xi\to \ell_\infty$ be any isometric embedding. Obviously $ji \in \complement (\mathcal{S}_\xi$-$\mathfrak{Sl}_1)$, since the canonical $c_{00}$ basis is an $\ell_1^\xi$ spreading model in $\ell_1$ and $X_\xi$.  We claim that $j i$ is universal for $\complement \mathcal{S}_\xi$-$\mathfrak{Sl}_1$

Let $B:X\to Y$ be an operator preserving an $\ell_1^\xi$ spreading model. Then either $B$ preserves a copy of $\ell_1$ or $B$ preserves a weakly null $\ell_1^\xi$ spreading model.  If $(x_n)_{n=1}^\infty$ is such that $(x_n)_{n=1}^\infty$ and $(Bx_n)_{n=1}^\infty$ are equivalent to the $\ell_1$ basis, then the maps $C:\ell_1\to X$, $A_1:[Bx_n: n\in \nn]\to X_\xi$ are such that $A_1BC=i$. If $A:E\to \ell_\infty$ is any extension of $jA_1$ (using the injectivity of $\ell_\infty$) , $ABC=ji$.

Suppose that $(x_n)_{n=1}^\infty\subset X$ is a weakly null $\ell_1^\xi$ spreading model in $X$ and that $(Bx_n)_{n=1}^\infty$ is an $\ell_1^\xi$ spreading model in $Y$. By Theorem \ref{the dichotomy}, after passing to a subsequence, we may assume $(Bx_n)_{n=1}^\infty$ is $\mathcal{S}_\xi$ unconditional. Since $(Bx_n)_{n=1}^\infty$ is an $\ell_1^\xi$ spreading model, there exists a constant $C>0$ such that $$\|\sum_{n\in E} a_n Bx_n\|\geqslant C \sum_{n\in E} |a_n|$$ for all $E\in \sss_\xi$ and scalars $(a_n)_{n\in E}$.    Since $(Bx_n)_{n=1}^\infty$ is $\mathcal{S}_\xi$ unconditional, there exists $C'>0$ such that $$\|\sum_{n=1}^\infty a_n Bx_n\|\geqslant C' \sup_{E\in \sss_\xi}\|\sum_{n\in E} a_n Bx_n\|$$ for any scalar sequence $(a_n)_{n=1}^\infty \in c_{00}$.   Now $$\|\sum_{n=1}^\infty a_n Bx_n\| \geqslant CC'\sup_{E\in \sss_\xi}\sum_{n\in E} |a_n|= CC'\|\sum_{n=1}^\infty a_n e_n\|_{X_\xi}$$ for any $(a_n)_{n=1}^\infty \in c_{00}$.  From this it follows that $C:\ell_1\to X$, $A_1:[Bx_n: n\in \nn]\to X_\xi$ given by $Ce_n=x_n$, $A_1Bx_n=e_n$ satisfy $A_1BC=i$. Then if $A:Y\to \ell_\infty$ is any extension of $jA_1$, $ABC=ji$.

$(ii)$ Let $K=\{\sum_{i\in E} e_i: E\in \mathcal{S}_\xi\}\subset c_{00}$. Here, $\sum_{i\in \varnothing } e_i=0$.  Let us  define $[\cdot]:c_{00}\to \rr$ by $$[x]=\inf\bigl\{\sum_{i=1}^n |a_i|:n\in\nn,  x=\sum_{i=1}^n a_i x_i, x_i\in K\}.$$ Then $[\cdot]$ is a norm on $c_{00}$.   Let $Y_\xi$ be the completion of $c_{00}$ with respect to this norm.  Now fix $x\in c_{00}$ and $j\in \nn$ such that $|e^*_j(x)|=\|x\|_{c_0}$.  Suppose $x=\sum_{i=1}^n a_i x_i$, $x_i\in K$.  Note that \begin{align*} \|x\|_{c_0} & = |e^*_j(\sum_{i=1}^n a_i x_i)| \leqslant \sum_{i=1}^n  |a_i||e^*_j(x_i)| \leqslant \sum_{i=1}^n |a_i|. \end{align*} From this it follows that $[\cdot] \geqslant \|\cdot\|_{c_0}$ on $c_{00}$.  In particular, it is easy to see that $[e_i]=1$ for all $i\in \nn$. Moreover, if $F\subset \nn$ is finite and $x=\sum_{i=1}^n a_i x_i$ for some $x_1, \ldots, x_n\in K$, then $$Fx=\sum_{i=1}^n a_i Fx_i$$ and $Fx_i\in K$. From this it follows that $[Fx]\leqslant [x]$, and the canonical $c_{00}$ basis is normalized and $1$-suppression unconditional in $Y_\xi$.  

By standard arguments, for any $E\in \mathcal{S}_\xi$ and any scalars $(a_i)_{i\in E}\subset [-1,1]$, $$\sum_{i\in E} a_i e_i \in \text{co}\Bigl( \sum_{i\in F} e_i - \sum_{i\in G} e_i: F, G\subset E\Bigr) \subset 2 B_{Y_\xi}.$$  From this it follows that the $c_{00}$ basis is a $c_0^\xi$ spreading model in $Y_\xi$. Now let $j:Y_\xi\to \ell_\infty$ be the canonical inclusion, so that $j$ preserves the $c_0^\xi$ spreading model $(e_i)_{i=1}^\infty$. 

Now suppose that $B:X\to Y$ is an operator and $(x_i)_{i=1}^\infty\subset X$ is a $c_0^\xi$ spreading model such that $(Bx_i)_{i=1}^\infty$ is also a $c_0^\xi$ spreading model.  Since $\xi>0$, $(Bx_n)_{n=1}^\infty$ must be weakly null and seminormalized. We may pass to a subsequence and relabel to assume that $(Bx_i)_{i=1}^\infty$ is seminormalized and basic. Then the map $Bx_i\mapsto e_i\in \ell_\infty$ extends to a bounded operator from $[Bx_i: i\in \nn]$ into $\ell_\infty$. By injectivity, we may extend this to a bounded, linear operator $A:Y\to \ell_\infty$ such that $ABx_i=e_i$.  Now define $C:c_{00}\to X$ by $C\sum_{i=1}^\infty a_i e_i= \sum_{i=1}^\infty a_i x_i$. Since $(x_i)_{i=1}^\infty$ is a $c_0^\xi$ spreading model, there exists a constant $c$ such that $\|\sum_{i\in E} x_i\|\leqslant c$ for all $E\in \mathcal{S}_\xi$.  In particular, if $x\in K$, $\|Cx\| \leqslant c$.  Now for any $x\in c_{00}$, if we write $x=\sum_{i=1}^n a_i x_i$ with $x_i\in K$, then $$\|Cx\| \leqslant \sum_{i=1}^n |a_i|\|Cx_i\|\leqslant c\sum_{i=1}^n |a_i|.$$   Taking the infimum over all such decompositions $x=\sum_{i=1}^n a_i x_i$, we deduce that $\|Cx\|\leqslant c[x]$ for all $x\in c_{00}$.  It then follows that $C:c_{00}\to X$ extends to an operator which we still call $C$ from $Y_\xi$ to $X$ such that $Ce_i=x_i$.  Then $j:Y_\xi\to \ell_\infty$ factors through $B:X\to Y$ by $ABC=j$.

$(iii)$ We first recall that $\mathcal{S}_\xi$ is treated both as a set of subsets of $\nn$ and as a set of finite sequences. As a set of sequences, $\mathcal{S}_\xi$ is a tree of  order $\omega^\xi+1$ and $\mathcal{S}_\xi^\zeta$ is also regular for every ordinal $\zeta$. In particular, $\mathcal{S}_\xi^{\omega^\xi}=\{\varnothing\}$. Let $T_\xi=\mathcal{S}_\xi\setminus \{\varnothing\}$ and define $L\subset c_{00}(T_\xi)$ by $$L=\bigl\{\sum_{E\preceq F\preceq G} e_F: E\preceq G, E,G\in T_\xi\bigr\}.$$  Let $Z_\xi$ be the completion of $c_{00}(T_\xi)$ with respect to the norm $$[x]=\inf \big\{\sum_{i=1}^n |a_i|: n\in \nn, x=\sum_{i=1}^n a_i x_i, x_i\in L\bigr\}.$$    Arguing as in $(ii)$, it is easy to see that the basis $(e_E)_{E\in T_\xi}$ is a normalized, $1$-unconditional basis for $Z_\xi$ and  $(2^{-1}e_F)_{F\preceq G}\in T_{\mathfrak{S}c_0}(I_{Z_\xi}, Z_\xi, Z_\xi, 2)$ for every $G\in \mathcal{S}_\xi$. From this it follows that $(2^{-1}e_F)_{F\preceq G}\in T_{\mathfrak{S}c_0}(I_{Z_\xi}, Z_\xi, Z_\xi, 2)^\zeta$ for each $G\in \mathcal{S}_\xi^\zeta$, and $\varnothing\in T_{\mathfrak{S}c_0}(I_{Z_\xi}, Z_\xi, Z_\xi, 2)^{\omega^\xi}$. Therefore $r_{c_0}(Z_\xi, Z_\xi, I_{Z_\xi})>\omega^\xi$.  Now let $j:Z_\xi\to \ell_\infty(T_\xi)$ be the canonical inclusion.

Let $(E_n)_{n=1}^\infty$ be some enumeration of $\mathcal{S}_\xi\setminus \{\varnothing\}$ such that if $E_m\prec E_n$, $m<n$. Fix a sequence $(\ee_n)_{n=1}^\infty$ of positive numbers such that $\prod_{n=1}^\infty (1+\ee_n)<2$. Define $h:\mathcal{S}_\xi\to [0, \omega^\xi]$ by $h(E)=\max\{\zeta\leqslant \omega^\xi: E\in \mathcal{S}_\xi^\zeta\}$.  Note that since $\xi\geqslant \omega$, $1+\xi=\xi$, and $\omega^{1+\xi}=\omega^\xi$. Assume that $B:X\to Y$ is such that $r_{c_0}(X,Y,B)>\omega^\xi$ 
and fix $K$ such that $o(T_{\mathfrak{S}c_0}(B,X,Y,K))>\omega^\xi$.  We now define a collection $(x_E)_{E\in T_\xi}\subset X$ such that if $1\leqslant k_1<\ldots <k_r$ are such that $E_{k_1}\prec \ldots \prec E_{k_r}$, $(x_{E_{k_i}})_{i=1}^r\in T_{\mathfrak{S}c_0}(B,X,Y,K)^{\omega h(E_{k_r})}$,
as well as finite subsets $S_1, S_2, \ldots$ of $B_{Y^*}$ such that for any $n\in \nn$, any  $y\in \text{span}\{Bx_{E_1}, \ldots, Bx_{E_n}\}$, $$(1+\ee_n)\max_{y^*\in S_{n+1}}|y^*(y)| \geqslant \|y\|.$$   First, since $\omega h(E_1)<\omega^{1+\xi}=\omega^\xi$, we may fix any $x_{E_1}\in X$ such that $(x_{E_1})\in T_{\mathfrak{S}c_0}(B,X,Y,K)^{h(E_1)}$. Now assume that $x_{E_1}, \ldots, x_{E_n}$ have been chosen.   Now fix a finite set $S_{n+1}\subset B_{Y^*}$  such that for  any  $y\in \text{span}\{Bx_{E_1}, \ldots, Bx_{E_n}\}$, $$(1+\ee_n)\max_{y^*\in S_{n+1}}|y^*(y)| \geqslant \|y\|.$$ Let $l=|S_{n+1}|+1$. Fix $k_1<\ldots <k_r$ such that $E_{k_1}\prec \ldots \prec E_{k_r}\prec E_{n+1}$, and that every non-empty, proper initial segment of $E_{n+1}$ lies in $\{E_{k_1}, \ldots, E_{k_r}\}$.    If $|E_{n+1}|=1$, we let $r=0$, $E_{k_r}=\varnothing$, and $(x_{E_i})_{i=1}^r:=\varnothing$.

By hypothesis, $$(x_{E_i})_{i=1}^r\in T_{\mathfrak{S}c_0}(B, X, Y, K)^{\omega h(E_{k_r})}.$$   Since $E_{k_r}\prec E_{n+1}$, $h(E_{k_r})\geqslant h(E_{n+1})+1$ and $\omega h(E_{k_r})\geqslant \omega h(E_{n+1})+\omega>\omega h(E_{n+1})+l$.  This means there exists a sequence $(w_i)_{i=1}^l$ such that $$(x_{E_1}, \ldots, x_{E_{k_r}}, w_1, \ldots, w_l)\in T_{\mathfrak{S}c_0}(B,X,Y,K)^{\omega h(E_{n+1})}.$$   
Simply by comparing dimensions, we may fix $w=\sum_{i=1}^l a_i w_i$ such that $\max_{1\leqslant i\leqslant l}|a_i|\neq 0$ and $w\in \cap_{y^*\in S_{n+1}} \ker(B^*y^*)$. By scaling, we may assume $\max_{1\leqslant i\leqslant l}|a_i|=1$. Let $x_{E_{n+1}}=w$.  This completes the recursive construction.

Now arguing as in the proof of Proposition \ref{factfact}, $(Bx_{E_n})_{n=1}^\infty$ is a Schauder basis for its closed span. It follows from the definition of $T_{\mathfrak{S}c_0}(B,X,Y,K)$ that $(x_{E_n})_{n=1}^\infty$ is seminormalized. More precisely, $1/K\leqslant \|Bx_{E_n}\| \leqslant \|B\|$. From this and injectivity of $\ell_\infty(T_\xi)$, there exists a bounded, linear map $A:Y\to \ell_\infty(T_\xi)$ such that $ABx_E=e_E$. Now define $C:c_{00}(T_\xi)\to X$ by $C\sum_{E\in T_\xi} a_E e_E=\sum_{E\in T_\xi}a_E x_E$. If $x=\sum_{E\preceq F\preceq G} e_F\in L$, then $\|Cx\|\leqslant 1$, since by construction $(x_F)_{E\preceq F\preceq G}\in T_{\mathfrak{S}c_0}(B,X,Y,K)$.   If $x=\sum_{E\in T_\xi}a_E e_E\in c_{00}(T_\xi)$ and $x=\sum_{i=1}^n a_i x_i$, $x_i\in L$, then by the triangle inequality, $$\|Cx\|\leqslant \sum_{i=1}^n |a_i|\|Cx_i\|\leqslant \sum_{i=1}^n |a_i|.$$  Taking the infimum over all such representations of $x$, we deduce that $C:(c_{00}(T_\xi), [\cdot])\to X$ has norm at most $1$ and extends by density to a norm at most $1$ operator $C:Z_\xi\to X$ such that $Ce_E=x_E$. Then $ABC=j$.

\end{proof}

\section{Universal Operators: Negative Results}

The main theorem of this section is a list of ideals whose complements do not admit universal operators. Items $(i)$ and $(ii)$ in Theorem \ref{soo bad}  are quantitative improvements of the results of Oikhberg \cite{Oik-Positivity} in which he showed that $\complement \mathfrak{S}$ does not have a universal operator while $\complement \mathfrak{FS}$ does. Essentially, we show that a universal factoring operator for $\complement \mathfrak{S}$ must be in the lowest complexity outside of $\mathfrak{FS}$.

\begin{theorem} \begin{enumerate}[(i)]

\item Any operator which factors through every non-strictly singular operator belongs to the class $\mathfrak{S}^2$. In particular, the complement of $\mathfrak{S}^\xi$ has a universal factoring operator if and only if $\xi=1$ (in this case $\mathfrak{S}^1=\mathfrak{FS}$).

\item Any operator which factors through every non-strictly singular operator belongs to the class $\mathcal{S}_1$-$\mathfrak{S}$. 

\item For each $0<\xi<\omega_1$, there is no $\complement \mathfrak{BS}^\xi$-universal factoring operator in $\complement \mathfrak{BS}^\xi$. In particular, $\complement \mathfrak{BS}$ does not have a universal operator. 
\item For each $0<\xi<\omega_1$, there is no $\complement \mathfrak{wBS}^\xi$-universal factoring operator in $\complement \mathfrak{wBS}^\xi$.
In particular $\complement \mathfrak{wBS}$ does not have a universal operator.\end{enumerate}

\label{soo bad}
\end{theorem}

\begin{proof}

 $(i)$ To obtain a contradiction, assume that $B:X\to Y$ lies in $\complement \mathfrak{S}^2$ and factors through every non-strictly singular operator. Then there exist operators $C_\infty:c_0\to X$, $A_\infty:Y\to c_0$ and $C_1:\ell_1\to X$, $A_1:Y\to \ell_1$ such that $A_\infty C_\infty= A_1C_1=B$. Fix $K>1$ such that $o(T_\mathfrak{S}(B,X,Y,K))>\omega^2$. Fix $\ee>0$ such that $$2\|A_\infty\|K \ee < \|A_1\|^{-2}K^{-2}-2\ee.$$    Fix $n\in \nn$ such that $$\|A_\infty\|K\Bigl(\|C_\infty\|+2n\ee\Bigr) < \|A_1\|^{-1}K^{-1}\Bigl(n \|A_1\|^{-1}K^{-1} - 2n\ee\Bigr).$$  
 We will select $x_1, \ldots, x_n\in S_X$ and block sequences $(u_i)_{i=1}^n$, $(v_i)_{i=1}^n$ in $c_0$, $\ell_1$, respectively, such that for each $1\leqslant i\leqslant n$, \begin{enumerate}[(i)]\item  $(x_j)_{=1}^i\in T_\mathfrak{S}(B,X,Y,K)^{\omega(n-i)}$,  \item  $\|u_i-C_\infty x_i\|<\ee$,  $\|v_i-C_1 x_i\|<\ee$. 
 \end{enumerate} We note that since $o(T_\mathfrak{S}(B,X,Y,K))>\omega^2$, there exists $x_1\in S_X$ such that $(x_1)\in T_\mathfrak{S}(B,X,Y,K)^{\omega (n-1)}$. Next, assume $x_1, \ldots, x_k$ and block sequences $(u_i)_{i=1}^k$, $(v_i)_{i=1}^k$ have been chosen, where each $u_i$ and $v_i$ have finite supports. Fix $m\in \nn$ such that $u_1, \ldots, u_k \in \text{span}\{e_j: j\leqslant m\}\subset c_0$, $v_1, \ldots, v_k\subset \text{span}\{e_j: j\leqslant m\}\subset \ell_1$  have been chosen. Note that since $(x_1, \ldots, x_k)\in T_\mathfrak{S}(B,X,Y,K)^{\omega(n-k+1)}$,  there exists a sequence $(z_i)_{i=1}^{2m+1}$ such that $(x_1, \ldots, x_k, z_1, \ldots, z_{2m+1})\in T_{\mathfrak{S}}(B,X,Y,K)^{\omega(n-k)}$. By a dimension argument, there exists a norm $1$ vector $x_{k+1}\in \text{span}\{z_1, \ldots, z_{2m+1}\}$ such that $C_\infty x_{k+1}\in \overline{\text{span}}\{e_j: j>m\}\subset c_0$ and $C_1 x_{k+1}\in \overline{\text{span}}\{e_j:j>m\}\subset \ell_1$. We may now fix vectors $u_{k+1}\in \text{span}\{e_j:j>m\}$ and $v_{k+1}\in \text{span}\{e_j: j>m\}$ such that $\|C_\infty x_{k+1}-u_{k+1}\|<\ee$ and $\|C_1 x_{k+1}-v_{k+1}\|<\ee$. This completes the recursive construction. 
 
 Now note that for each $1\leqslant i\leqslant n$, $$\|u_i\|\leqslant \|C_\infty\| \|x_i\|+\ee=\|C_\infty\|+\ee$$ and $$\|v_i\|\geqslant \|C_1 x_i\|-\ee \geqslant \|A_1\|^{-1}\|A_1C_1x_i\|-\ee \geqslant \|A_1\|^{-1}K^{-1} \|x_i\|-\ee = \|A_1\|^{-1}K^{-1}-\ee. $$

\noindent Then \begin{align*} \|\sum_{i=1}^n x_i\| & \leqslant K \|\sum_{i=1}^n Bx_i\| \leqslant  \|A_\infty\|K \|\sum_{i=1}^n C_\infty x_i\| \\ & \leqslant \|A_\infty\|K \Bigl(\|\sum_{i=1}^n u_i\|+ \ee n\Bigr) \\ & \leqslant \|A_\infty\|K\Bigl(\|C_\infty\|+2 \ee n\Bigr),\end{align*} while \begin{align*} \|\sum_{i=1}^n x_i\| & \geqslant K^{-1} \|\sum_{i=1}^n Bx_i\| \geqslant \|A_1\|^{-1}K^{-1}\Bigl(\|\sum_{i=1}^n v_i\| -n \ee\Bigr) \\ & \geqslant \|A_1\|^{-1}K^{-1}\Bigl(n\|A_1\|^{-1}K^{-1}-2n \ee\Bigr).\end{align*} By our choice of $\ee$ and $n$, these two inequalities yield a contradiction. 

$(ii)$ We note that the following proof is a quantified version of Oihkberg's proof that there is not $\complement \mathfrak{S}$-universal operator. 

To obtain a contradiction, assume $B:X\to Y$ lies in $\complement \mathcal{S}_1$-$\mathfrak{S}$ and factors through every non-strictly singular operator. Then $B$ factors through both $I_{c_0}$ and $I_{\ell_1}$.  Fix operators $C_\infty:X\to c_0$, $A_\infty:c_0\to Y$ and $C_1:X\to \ell_1$ and $A_1:\ell_1\to Y$ such that $A_\infty C_\infty=A_1C_1=B$.    Fix a normalized basic sequence $(x_n)_{n=1}^\infty\subset X$ and $c>0$ such that $\|A\sum_{i\in E} a_i x_i\|\geqslant c \|\sum_{i\in E} a_i x_i\|$ for all $E\in \mathcal{S}_1$ and all scalars $(a_i)_{i\in E}$.    By passing to a subsequence, we may assume $(C_\infty x_n)_{n=1}^\infty$ and $(C_1x_n)_{n=1}^\infty$ are coordinate-wise convergent in $c_0$ and $\ell_1$, respectively. Fix a  sequence $(\ee_n)_{n=1}^\infty$ such that $\sum_{n=1}^\infty \ee_n<\infty$ and note that we may select $p_1<p_2<\ldots$,  a seminormalized block sequence $(u_n)_{n=1}^\infty$ in $c_0$, and a seminormalized block sequence $(v_n)_{n=1}^\infty$ in $\ell_1$ such that for all $n\in \nn$, $$\|C_\infty(x_{p_{2n}}- x_{p_{2n+1}})- u_n\|<\ee_n$$ and $$\|C_1 (x_{p_{2n}}- x_{p_{2n+1}})-v_n\|<\ee_n.$$ From this it follows that $$\underset{n}{\lim\sup} \|C_\infty\frac{1}{n}\sum_{i=n+1}^{2n} (x_{p_{2i}}-x_{p_{2i+1}}) \|=0,$$ while $$\underset{n}{\lim\inf} \|C_1\frac{1}{n}\sum_{i=n+1}^{2n} (x_{p_{2i}}-x_{p_{2i+1}})\|>0.$$  But since $\cup_{i=n+1}^{2n} \{p_{2i}, p_{2i+1}\}\in \mathcal{S}_1$ for all $n\in \nn$, it follows that for all $n\in \nn$, \begin{align*} \|C_1 \frac{1}{n}\sum_{i=n+1}^{2n} (x_{p_{2i}}-x_{p_{2i+1}})\| & \leqslant \|C_1\|\|\frac{1}{n}\sum_{i=n+1}^{2n} (x_{p_{2i}}-x_{p_{2i+1}})\| \leqslant \|C_1\|c^{-1}\|B\frac{1}{n}\sum_{i=n+1}^{2n}(x_{p_{2i}}-x_{p_{2i+1}})\| \\ & \leqslant \|C_1\|c^{-1}\|A_\infty\|\|C_\infty\| \|\frac{1}{n}\sum_{i=n+1}^{2n} (x_{p_{2i}}-x_{p_{2i+1}})\|.  \end{align*}

This contradiction finishes $(ii)$.

$(iii)$ As observed in \cite{BC-Scand}, $\mathfrak{BS}^\xi = \mathfrak{W}\cap \mathcal{S}_\xi$-$\mathfrak{Sl}_1$. Since neither of the classes $\mathfrak{W}$, $\mathcal{S}_\xi$-$\mathfrak{Sl}_1$ is contained in the other, Remark \ref{get away from me} yields that the complement of the intersection cannot have a universal factoring operator. In order to see that neither of these classes is contained in the other, we observe that $I_{c_0}\in \mathcal{S}_\xi$-$\mathfrak{Sl}_1\cap \complement \mathfrak{W}$ for any $0<\xi<\omega_1$, while for any $0<\xi<\omega_1$, $I_{T_\xi}\in \mathfrak{W}\cap \complement \mathcal{S}_\xi$-$\mathfrak{Sl}_1$.

$(iv)$ The complement of $\mathfrak{wBS}^\xi$ is the class of those operators $A:X\to Y$ such that there exists a weakly null sequence $(x_n)_{n=1}^\infty$ such that $(Ax_n)_{n=1}^\infty$ is an $\ell_1^\xi$ spreading model. To that end, if $U:E\to F$ is a universal factoring operator for $\complement \mathfrak{wBS}^\xi$, then we may fix a weakly null sequence $(u_n)_{n=1}^\infty \subset B_E$ such that $(Uu_n)_{n=1}^\infty$, and therefore $(u_n)_{n=1}^\infty$, is an $\ell_1^\xi$ spreading model.  For each $\xi<\gamma<\omega_1$, canonical basis of $T_\gamma$ is an $\ell_1^\gamma$ spreading model. Since $\xi<\gamma<\omega_1$, there exists $m\in \nn$ such that $m\leqslant E\in \mathcal{S}_\xi$ implies $E\in \mathcal{S}_\gamma$. From this it follows that any $\ell_1^\gamma$ spreading model is also an $\ell_1^\xi$ spreading model.  The canonical basis of $T_\gamma$ is therefore an $\ell_1^\xi$ spreading model, and  $U$ factors through the identity on $T_\gamma$. Let $C:E\to T_\gamma$, $A:T_\gamma\to F$ be such that $U=AI_{T_\gamma}C$. Then since $(u_n)_{n=1}^\infty$ is weakly null, some subsequence of $(Cu_n)_{n=1}^\infty$ is equivalent to a seminormalilzed block sequence in $T_\gamma$. Since every seminormalized block sequence in $T_\gamma$ is an $\ell_1^\gamma$ spreading model, there exists a natural number $k_\gamma$ such that some subsequence of $(u_n)_{n=1}^\infty$ is a $1/k_\gamma$-$\ell_1^\gamma$ spreading model. Then there exists $k\in \nn$ such that $\sup \{\gamma\in (\xi, \omega_1): k_\gamma=k\}=\omega_1$, whence by a standard overspill argument, some subsequence of $(u_n)_{n=1}^\infty$ is equivalent to the $\ell_1$ basis. This contradicts the fact that $(u_n)_{n=1}^\infty$ is assumed to be weakly null.

\end{proof}

\section{The Standard Borel Space $\mathcal{L}$, Genericity and Universality}

In this section we present a descriptive set theoretic approach for studying factorization and genericity for classes of operators. In particular, our main tool is the coding $\mathcal{L}$ of all bounded linear operators between separable Banach spaces defined in \cite{BFr-JFA}. While the coding of all separable Banach spaces $\textbf{SB}$ has been extensively studied in the context of embedding results and universality, the Standard Borel space $\mathcal{L}$ is a relatively new object. This section contains many results that are analogous to those known for $\textbf{SB}$ as well as results for operator ideals that deal with both the existence of a universal operator for the complement of an ideal and genercity of classes of operators.

This section is organized as follows: First we recall several known embedding and genericity results for $\textbf{SB}$ and introduce new relations on $\mathcal{L}$. We then state our complexity and genericity results for operator ideals restricted to $\mathcal{L}$. After this, we recall the necessary descriptive set theoretic notions and prove several lemmas. The final few  subsections contain the proofs of the main theorems.


As promised, before we state our results for $\mathcal{L}$, we recall what is known for $\textbf{SB}$. Consider the following subspaces of $\textbf{SB}$: $\textbf{REFL}=\text{Space}(\mathfrak{W})\cap \textbf{SB}$, the separable,  reflexive spaces; $\textbf{SD}=\text{Space}(\mathfrak{D})\cap \textbf{SB}$, spaces with separable dual; $\textbf{NC}_E=\text{Space}(\mathfrak{S}E)\cap \textbf{SB}$, separable spaces not containing an isomorphic copy of $E\in \textbf{SB}$; $\textbf{NU}=\text{Space}(\mathfrak{S}C(2^\mathbb{N}))\cap \textbf{SB}$, denotes the class of separable spaces not containing an isomorphic copy of every separable Banach space. Each of the above subsets is $\Pi_1^1$-complete as a subset of $\textbf{SB}$. Also, for a given $X$ in $\textbf{SB}$, the collection of all $Y\in \textbf{SB}$ that embed into $X$ is analytic. These two facts yield that none of these classes contains a space into which every other space in the class isomorphically embeds. In a series of papers and a monograph \cite{ADo-Advances,Bos-Thesis,Bos-FunD,Do-TAMS,Do-Studia,DoFe-Fund}, several authors studied different kinds of genericity for these classes of spaces. For completeness, we recall these notions below. 


\begin{definition}
Let $C \subset \textbf{SB}$. 
\begin{enumerate} 
    \item The set $C$ is \emph{Bourgain generic} if whenever every member of $C$ isomorphically embeds into $X$, then every element of $\textbf{SB}$ isomorphically embeds into $X$.
    \item A set $C$ is called \emph{Bossard generic} if whenever whenever $A\subset \textbf{SB}$ is  an analytic set which contains an isomorphic copy of each member of $C$, then there exists $Y\in A$ which contains isomorphic copies of all members of  $\textbf{SB}$. 
    \item The set $C$ is \emph{strongly bounded} if for each analytic subset $A$ of $C$, there exists an $X \in C$ which contains isomorphic copies of each member of $A$.
\end{enumerate}
\end{definition}

The term Bourgain generic (see \cite{ADo-Advances}) refers to the fact that Bourgain proved that $\textbf{REFL}$ is generic in this sense \cite{Bou-PAMS}. Bossard generic was named for the analogous reason. From the observation in the previous paragraph, it is easy to see that a class $C$ is Bourgain generic if it is Bossard generic. 
P. Dodos, by showing that $\textbf{NU}$ is strongly bounded, showed that the reverse implication holds \cite{Do-TAMS}.

\begin{theorem}
For each Banach space $E$ with a basis, $\textbf{\emph{NC}}_{E}$ is Bossard generic. Moreover, the classes $\textbf{\emph{REFL}}$, $\textbf{\emph{SD}}$, $\textbf{\emph{NC}}_E$ ($E$ minimal, infinite dimensional, not containing $\ell_1$), and $\textbf{\emph{NU}}$ are Bossard-generic and strongly bounded. 
\label{about the spaces}
\end{theorem}

We note there are several interesting papers that compute complexities of other classes of Banach spaces \cite{Braga-Czech,Braga-JMAA} and other papers that consider isometric  embeddings \cite{Kurka-StudiaAmal,Kurka-StudiaNU}.

As an extension of the the above list, we have the following result regarding $\textbf{SB}$. The proof of this theorem follows from known techniques but we could not find this statement in the literature. We give the prove of this theorem in subsection 5.4.

\begin{theorem} The following subsets of $\textbf{\emph{SB}}$ are Bossard generic:
\begin{enumerate}[(i)]
    \item $\textbf{\emph{DP}}=\text{\emph{Space}}(\mathfrak{DP})\cap \emph{\textbf{SB}}$, spaces with the Dunford-Pettis property;
    \item $\textbf{\emph{V}}=\text{\emph{Space}}(\mathfrak{V})\cap \emph{\textbf{SB}}$, Schur spaces;
    \item $ \textbf{\emph{BS}}=\text{\emph{Space}}(\mathfrak{BS})\cap \emph{\textbf{SB}}$ spaces with the Banach-Saks property;
    \item $\textbf{\emph{NC}}_E \cap \emph{\textbf{SB}}$ for a separable Banach space $E$;
    \item $\textbf{\emph{wBS}}=\text{\emph{Space}}(\mathfrak{wBS})\cap \emph{\textbf{SB}}$ spaces with the weak Banach-Saks property.
\end{enumerate}
\label{spatial}
\end{theorem}

The main results of this section are to prove Bossard genericity and strong boundedness for many of previously mentioned classes of operators in $\mathcal{L}$. Studying genericity and strong boundedness for subsets of $\mathcal{L}$ is a potentially richer theory since each of the classes in the Theorem \ref{about the spaces} yields a natural operator ideal analogue, and there are many proper ideals (e.g. $\mathfrak{S}$) which do not have  interesting spatial analogues.

 In order to define these notions for $\mathcal{L}$, we need the proper analogue of `isomorphically embeds' for operators. Fortunately, we already isolated the correct property: \emph{$A \in \mathcal{L}$ factors through a restriction of an operator $B \in \mathcal{L}$}. 
 Note that all members of  $\mathcal{L}$ factor through a restriction of the identity operator on $C(2^\mathbb{N})$, and so in this sense $id_{C(2^\mathbb{N})}$ in $\mathcal{L}$ plays the role of $C(2^\mathbb{N})$ in $\textbf{SB}$. We also need introduce a notion of `isomorphic' operators which generalizes isomorphism for spaces.

\begin{definition}
 Let $A\in \mathfrak{L}(X,Y)$ and $B \in \mathfrak{L}(X',Y')$. Then $A$ is \emph{isomorphic} to $B$ and we write $A \cong B$ if there exist isomorphisms $i:X\to X'$ and $j:Y\to Y'$ such that

\[
\begin{tikzcd}
X \arrow{r}{A} \arrow[shift right, swap]{d}{i}& Y \arrow[shift right, swap]{d}{j}\\
X' \arrow{r}{B} \arrow[shift right, swap]{u}{i^{-1}} & Y' \arrow[shift right, swap]{u}{j^{-1}}
\end{tikzcd}
\]
commutes both ways.
\end{definition}

We can now introduce the definitions of generic and strongly bounded for collections of operators.

\begin{definition}
Let $\mathcal{C} \subset \mathcal{L}$. 
\begin{enumerate}
\item The set $\mathcal{C}$ is \emph{Bossard generic} if whenever $\mathcal{A} \subset \mathcal{L}$ is analytic and contains an isomorphic copy of each member of $\mathcal{C}$,  there exists an $A \in \mathcal{A}$ so that every $B \in \mathcal{L}$ factors through a restriction of $A$.
\item The set $\mathcal{C}$ is \emph{strongly bounded} if for every analytic subset $\mathcal{A}$ of $\mathcal{C}$ there is a $B \in \mathcal{C}$ such that every $A \in \mathcal{A}$ factors through a restriction of $B$. 
\end{enumerate}
\end{definition}

In section 2 we defined the notion of Bourgain generic for operators.  However, as it is a formally weaker notion than Bossard generic (and they likely coincide), our results will focus on Bossard-genericity.

We now state the main results of this section.

\begin{theorem}
Fix an ordinal  $0<\xi<\omega_1$ and $1 \leqslant p <\infty$.  Intersecting any one of the following ideals with $\mathcal{L}$ yields a $\Pi_1^1$-complete,  Bossard generic  subset  of $\mathcal{L}$: 
\begin{enumerate}
    \item weakly compact operators $\mathfrak{W}$;
    \item Asplund operators $\mathfrak{D}$;
    \item strictly singular operators $\mathfrak{S}$;
    \item  $E$-singular $\mathfrak{S}E$, for any infinite dimensional $E$;
    \item $\xi$-Banach-Saks $\mathfrak{BS}^\xi$ and $\mathcal{S}_\xi$-weakly compact $\mathcal{S}_\xi$-$\mathfrak{W}$;
    \item  $\mathcal{S}_\xi$-$E$-singular $\mathcal{S}_\xi$-$\mathfrak{S}E$ for $E$ with a $1$-spreading basis;
    \item $\mathcal{S}_\xi$-strictly singular $\mathcal{S}_\xi$-$\mathfrak{S}$.
    \item $\xi$-weak Banach Saks operators $\mathfrak{wBS}^\xi$.
\end{enumerate}
\label{lots of stuff}
\end{theorem}

\begin{theorem}
The completely continuous $\mathfrak{V}$ and Dunford-Pettis $\mathfrak{DP}$ operator ideals are $\Pi^1_2$-hard and Bossard generic.
\label{it's complicated}
\end{theorem}

{\color{red} 

}


We also prove the following result concerning strongly bounded classes of operators.

\begin{theorem}
The weakly compact $\mathfrak{W}$ and Asplund operators $\mathfrak{D}$ in $\mathcal{L}$ are strongly bounded classes. 
\label{strongly bounded}
\end{theorem}
\begin{corollary}
The class $\mathfrak{W}^{super}$ and the classes $\mathfrak{D}_\xi$ for $\xi <\omega_1$ in $\mathcal{L}$ are not generic.
\label{not generic}
\end{corollary}

\begin{remark} In \cite{C-Fund}, a rank $\varrho$ was defined such that for an operator $A:X\to Y$, $\varrho(A)$ is an ordinal if and only if $A$ is weakly compact, $\varrho(A)\leqslant \omega$ if and only if $A$ is super weakly compact, and if $X$ is separable, $\varrho(A)<\omega_1$ if and only if $A$ is weakly compact.  The fact that $\mathfrak{W}^{super}$ is not generic which is stated in Corollary \ref{not generic} can be replaced by the more general fact that for each $\xi<\omega_1$, $\mathfrak{W}^\xi:=\{A: \varrho(A) \leqslant \xi\}$ is not generic. The proof is the same as for $\mathfrak{W}^{super}$.

\end{remark}

We will assume the previously stated results and prove the Corollary \ref{not generic}.

\begin{proof}
In \cite{BCFrWa-JFA} it is shown that $\mathfrak{W}^{super}\cap \mathcal{L}$ and $\mathfrak{D}_\xi\cap \mathcal{L}$ for $\xi <\omega_1$ are Borel subsets of $\mathfrak{W}$ and $\mathfrak{D}$, respectively. Therefore the conclusion follows from the fact that by Theorem \ref{strongly bounded}, these ideals are strongly bounded.
\end{proof}

\begin{remark}
In the authors' opinion the most interesting unsolved problem is whether $\mathfrak{S}$ is a strongly bounded. Since $\mathfrak{S}$ is a proper ideal, it has no (non-trivial) spatial analogue in $\textbf{SB}$. Consequently, any proof of this statement will probably require adapting, in a non-trivial way, many of the deep ideas from the work of Argyros and Dodos \cite{ADo-Advances} to a purely operator setting. We note that in our proof that $\mathfrak{W}$ and $\mathfrak{D}$ are strongly bounded, we use that $\textbf{REFL}$ and $\textbf{SD}$ are strongly bounded. 
\end{remark}

\subsection{Basic Descriptive Set Theory and the Definition of $\mathcal{L}$}

In this subsection we recall the definitions and some basic notions in Descriptive Set Theory (see \cite{Do-Book,Ke-book} for a detailed account). First recall that a Polish space is a separable, metrizable, topological space and a space $S$ endowed with a $\sigma$-algebra $\Sigma$ is called an standard Borel space (or simply standard space) if $\Sigma$ coincides with the $\sigma$-algebra of a Polish topology on $S$. The first standard space we consider is $\textbf{SB}$.
Recall that $\textbf{SB}$ denotes the set of closed subsets of $C(2^\nn)$ which are linear subspaces, endowed with the Effros-Borel structure.  We also recall the existence of a sequence $(d_n)_{n=1}^\infty$ of Borel functions from $\textbf{SB}$ into $C(2^\nn)$ such that for each $X\in \textbf{SB}$, $\{d_n(X): n\in \nn\}$ is a dense subset of $X$. We are only concerned with the Borel $\sigma$-algebra on $\textbf{SB}$, rather than the topology which generates it, so we will  often leave the topology unspecified or choose a specific topology as suits our needs. 

The following coding of $\mathcal{L}$ was defined in \cite{BFr-JFA} and further studied in \cite{BC-Scand,BCFrWa-JFA}. We let $\mathcal{L}$ denote the set of all triples $(X,Y, (y_n)_{n=1}^\infty)\in \textbf{SB}\times \textbf{SB}\times C(2^\nn)^\nn$ such that $y_n\in Y$ for all $n\in \nn$ and there exists $k\in \nn$ such that for every $n\in \nn$ and every collection $(q_i)_{i=1}^n$ of rational scalars, $$\|\sum_{i=1}^n q_i y_i\|\leqslant k \|\sum_{i=1}^n q_i d_i(X)\|.$$   The set $\mathcal{L}$ is a standard Borel subset of the Polish $\textbf{SB}\times \textbf{SB}\times C(2^\nn)^\nn$ endowed with the product Borel $\sigma$-algebra. Again, as it suits us, we may specify a topology on $\textbf{SB}$, after which we will assume $\mathcal{L}$ is endowed with the corresponding product topology. Note that $\mathcal{L}$ is a coding of all separable operators between Banach spaces (see \cite{BCFrWa-JFA} for a detailed treatment). Here we simply point out that, if $X,Y\in \textbf{SB}$ and $A:X\to Y$ is a bounded linear operator, $(X,Y, (Ad_n(X))_{n=1}^\infty)\in \mathcal{L}$, and if $(X,Y, (y_n)_{n=1}^\infty)\in \mathcal{L}$, the function $f:\{d_n(X):n\in \nn\}\to Y$ given by $f(d_n(X))=y_n$ extends uniquely to a continuous, linear operator from $X$ into $Y$.



For convenience, we often write $(X, Y, A)$ instead of $(X, Y, (y_n)_{n=1}^\infty)$. For $X \in \textbf{SB}$, $id_{X}$ will denote the identity operator on $X$. We isolate the following easy lemma which observes the connection between $\textbf{SB}$ and $\mathcal{L}$

\begin{lemma}
The map $\Psi:\textbf{\emph{SB}}\to \mathcal{L}$ by $\Psi(X)= (X,X,(d_n(X)))$ is a Borel isomorphism onto its range. In particular,  $\Psi(\textbf{\emph{SB}})=:\textbf{\emph{Sp}}$ is Borel in $\mathcal{L}$.
\label{cash money}
\end{lemma}

A subset $A$ of a standard Borel space $S$ is \emph{analytic} (or $\Sigma_1^1$) if there is a standard Borel space $S'$, a Borel subset $B$ of $S'$, and a Borel map $f:S' \to S$ such  that $f(B)=A$ (i.e. $A$ is the Borel image of a Borel set). A subset of a standard space is \emph{coanalytic} (or $\Pi_1^1 $) if it is the complement of an analytic set.  Lusin's theorem states that a set that is both analytic and coanalytic is Borel. A subset of a Polish space is $\Sigma^1_2$ if is is the Borel image of a $\Pi_1^1$ set. Likewise a subset of a Polish space is $\Pi_2^1$ is it is the complement of a $\Sigma_2^1$ space. A detailed explanation of the projective hierarchy can by found in \cite{Ke-book}.

Let $\Gamma$ be any class of sets in Polish spaces (e.g. $\Pi_1^1$). If $X$ is a standard Borel space let $\Gamma(X)$ be the subsets of $X$ in $\Gamma$.  A subset $B$ of a standard Borel space $S$ is said to be $\Gamma$-hard if for any standard
Borel space $S'$ and any $A \in \Gamma(S')$, there is a Borel function $f:S' \to S$ such that $x \in A$ if and only if $f(x) \in B$. If, in addition,
$B$ is in $\Gamma(S)$, then $B$ is said to be $\Gamma$-complete. The next lemma relates the complexity of a class of operators to the corresponding spaces.

\begin{lemma} Let $\Gamma$ be any class of sets in Polish spaces and $\mathcal{A}\subset \mathcal{L}$. If $\text{\emph{Space}}(\mathcal{A})$ in $\textbf{\emph{SB}}$ is $\Gamma$-complete, then $\mathcal{A}$ is $\Gamma$-hard.  \label{too hard}
\end{lemma}

\begin{proof}
Suppose that $\text{Space}(\mathcal{A})$ is $\Gamma$-complete. Then, by definition, for any standard Borel space $S$ and subset $A$ that is in $\Gamma$ there is a Borel map $f:S \to \textbf{SB}$ such that $x \in A$ if and only if $f(x) \in \text{Space}(\mathcal{A})$. By Proposition \ref{cash money}, the map $g:\textbf{SB} \to \mathcal{L}$ defined by $g(X) = (X,X, (d_n(X))_{n=1}^\infty)$ is Borel. The map $g \circ f : S \to \mathcal{L}$ is the desired reduction of $A$ to $\mathcal{A}$. This proves the claim.   
\end{proof}

 For the definition of $\Pi_1^1$-rank we refer the reader to \cite{Do-Book}. For our purposes we need the following facts. Recall that $Tr$ denotes the subset of $2^{2^{<\nn}}$ consisting of those subsets $T$ of $2^{<\nn}$ which are trees (that is, which contain all initial segments of their members), and $WF$ denotes the subset of $Tr$ consisting of those trees which are well-founded.

\begin{fact}
If $P$ is a Polish space and $C \subset P$, $f:P \to Tr$, is Borel, and $f^{-1}(WF)=C$,  then $C$ is $\Pi_1^1$ and $\phi(x)=o(f(x))$ defines a $\Pi_1^1$ rank on $A$.\end{fact}

We will also use the following fundamental properties of $\Pi_1^1$ ranks 

\begin{fact}
Let $P$ be a Polish space, $C$ be a $\Pi_1^1$ subset of $P$, and $\phi:C \to \omega_1$ be a $\Pi_1^1$-rank. Then
\begin{enumerate}
    \item for every $\xi<\omega_1$, the set $\{x : \phi(x) \leqslant \xi\}$ is Borel,
    \item if $A \subset C$ is analytic then $\sup\{\phi(x):x \in A\} <\omega_1$.
\end{enumerate}
\label{boundedness}
\end{fact}

\subsection{Complexity of Relations for $\mathcal{L}$}

The main purpose of this section is to prove the following proposition.

\begin{proposition} Let $\mathcal{B}\subset \mathcal{L}$ be Borel. Then each of the following classes is analytic. \begin{enumerate}[(i)]\item The class of operators factoring through a member of $\mathcal{B}$. \item The class of operators through which a member of $\mathcal{B}$ factors. \item The class of operators $A$ such that a member of $\mathcal{B}$ factors through a restriction of $A$.  \item The class of operators which factor through a restriction of a member of $\mathcal{B}$. \end{enumerate}
\label{lets have relations}
\end{proposition}

\begin{lemma} For $u\in C(2^\nn)$, let $D_u:\textbf{\emph{SB}}\to \rr$ be given by $D_u(X)=\inf_{x\in X} \|u-x\|$.  Then there exists a Polish topology $\tau$ on $\textbf{SB}$ whose Borel $\sigma$-algebra is the Effros-Borel $\sigma$-algebra and such that $D_u:(\textbf{\emph{SB}}, \tau)\to \rr$ is continuous for each $u\in C(2^\nn)$ and $d_n:(\textbf{\emph{SB}}, \tau)\to C(2^\nn)$ is continuous for each $n\in \nn$.

\label{beer}
\end{lemma}

\begin{proof} Let $\tau'$ be the coarsest topology on $\textbf{SB}$ such that $D_u:(\textbf{SB}, \tau')\to \rr$ is continuous for each $u\in C(2^\nn)$ (this topology is called the \emph{Beer topology}). Then this is a Polish topology whose Borel $\sigma$-algebra is the Effros-Borel $\sigma$-algebra \cite[Page 75]{Ke-book}.  By standard techniques \cite[pages 82, 83]{Ke-book}, there is a Polish topology $\tau$ on $\textbf{SB}$ such that $\tau'\subset \tau$, $\tau$ and $\tau'$ generate the same $\sigma$-algebra,  and such that for all $n\in \nn$, $d_n:(\textbf{SB }, \tau)\to C(2^\nn)$ is continuous. \end{proof}

We need the following lemma.
\begin{lemma}For $r\in \nn$, let $$R_r=\{(u,U,V,B): u\in U, (U,V,B)\in \mathcal{L}, \|B\|\leqslant r\}$$ and let $$P_r=\{((X,Y,A),(U,V,B))\in \mathcal{L}\times \mathcal{L}: \|B\|\leqslant r, A(X)\subset V\subset Y\}.$$   \begin{enumerate}[(i)]\item The set $R_r$ is Borel and the function $E:R_r\to C(2^\nn)$ given by $E(u,U,V,A)=Av$ is Borel. \item The set $P_r$ is Borel and the function $C(((X,Y,A), (U,V,B)))=(X,V, BA)$ is Borel.  \end{enumerate}
\label{ddbb}
\end{lemma}

\begin{proof}Assume $\textbf{SB}$ is topologized with the topology $\tau$ given in Lemma \ref{beer}. Then $R_r$ consists of those $(u,U, V, (v_n)_{n=1}^\infty)\in C(2^\nn)\times \mathcal{L}$ such that \begin{enumerate}[(i)]\item for any $s\in \nn$,  any scalar sequence $(a_i)_{i=1}^s$,  and any $s\in \nn$, $\|\sum_{i=1}^s a_i v_i\|\leqslant r \|\sum_{i=1}^s a_i d_i(U)\|$, \item $D_u(U)=0$, \end{enumerate} and $P_r$ consists of those $((X,Y,(y_n)_{n=1}^\infty ),(U,V,(v_n)_{n=1}^\infty)\in \mathcal{L}\times \mathcal{L}$ such that \begin{enumerate}[(i)]\item for any scalar sequence $(a_i)_{i=1}^n$,  any $s\in \nn$,and  $\|\sum_{i=1}^s a_i v_i\|\leqslant r \|\sum_{i=1}^s a_i d_i(U)\|$, \item for all $n\in \nn$, $D_{y_n}( V)=D_{d_n(V)}(Y)=0$.   \end{enumerate}

These are closed sets in $C(2^\nn)\times \mathcal{L}$, $\mathcal{L}\times \mathcal{L}$, respectively, where $\mathcal{L}$ has the product topology coming from $\tau$ and the norm topology of $C(2^\nn)$. Therefore $R_r$ and $P_r$ are Borel sets.

Now fix $(u_n, U_n, V_n, (v_{n,k})_{k=1}^\infty)_{n=1}^\infty\subset R_r$ converging to $(u, U, V, (v_k)_{k=1}^\infty)$.  Let $B_n:U_n\to V_n$ be the operator coming from $(U_n, V_n, (v_{n,k})_{k=1}^\infty)$   Fix $\ee>0$ and, to obtain a contradiction, assume $\inf_n \|B_nu_n-Bu\|\geqslant \ee$.   We may fix $k\in \nn$ such that $\|u-d_k(U)\|<\ee/9r$ and $m\in \nn$ such that for any $n\geqslant m$, $\|d_k(U_n)-d_k(U)\|<\ee/9r$, $\|v_{n,k}-v_k\|<\ee/3$, and $\|u_n-u\|<\ee/9r$.  Then for any $n\geqslant m$, $$\|d_k(U_n)-u_n\| \leqslant \|d_k(U_n)-d_k(U)\| + \|d_k(U)-u\| + \|u-u_n\| < 3(\ee/9r)= \ee/3r,$$ and \begin{align*} \|B_nu_n -Bu \| & \leqslant \|B_nu_n-B_n d_k(U_n)\| + \|B_nd_k(U_n)-Bd_k(U)\| +\|Bd_k(U)-Bu\| \\ & =  \|B_nu_n-B_n d_k(U_n)\| + \|v_{n,k}-v_k\| +\|Bd_k(U)-Bu\| \\ & <\ee/3 + \ee/3+ \ee/3 =\ee.  \end{align*} This contradiction shows the continuity of $E$ with respect to our topology of choice. Thus $E$ is Borel.

Now assume $((X_n, Y_n, (y_{n,k})_{k=1}^\infty), (U_n, V_n, (v_{n,k})_{k=1}^\infty))_{n=1}^\infty \subset P_r$ is a sequence which converges to $((X,Y, (y_k)_{k=1}^\infty), (U, V, (v_k)_{k=1}^\infty))$.    Let $A_n$ be the operator corresponding to $(X_n, Y_n, (y_{n,k})_{k=1}^\infty)$, and let $B_n$, $A$, $B$ be the operators corresponding to the other triples.  Note that $B_nA_n$, $BA$ are well-defined operators, and we may assign to each pair of triples some triple $(X_n, V_n, (w_{n,k})_{k=1}^\infty)$ and $(X, V, (w_k)_{k=1}^\infty)$.   Note that $w_{n,k}= B_nA_n d_k(X_n)= B_n y_{n,k}$ and $w_k=BAd_k(X)= B y_k$.    In order to see that the composition is continuous, we need to know that $X_n\to X$, $V_n\to V$, and $w_{n,k}\to w_k$ for each $k\in \nn$.   Of course, $X_n\to X$, $V_n\to V$ by assumption. Fix $k\in \nn$. Then by hypothesis, $y_{n,k}\to y_k$, whence by continuity of the evaluation, $B_ny_{n,k}\to By_k$. This shows continuity of $C$ with respect to our topology of choice, and $C$ is Borel. 
\end{proof}

We can now prove Proposition \ref{lets have relations}.

\begin{proof}[Proof of Proposition \ref{lets have relations}]

We define the following set of 4-tuples that will describe factorization of one operator through another.

$$\mathcal{F}=\{(X_i, Y_i, B_i)_{i=1}^4 \in \mathcal{L}^4: X_1=X_4,~ Y_1=X_2,~ Y_2=X_3,~Y_3=Y_4,~ B_3 B_2 B_1=B_4\}.$$ 

The following diagram illustrates the reason we care about $\mathcal{F}$. 

$$\begin{CD}  X_1=X_4 	@>B_4>>	Y_3=Y_4  \\ @VVB_1V			@AAB_3A \\ Y_1=X_2	@>B_2>>	Y_2=X_3\end{CD}$$

That is, $B_4:X_4\to Y_4 $ factors through $B_2:X_2\to Y_2$ if and only if there exist $(X_1, Y_1, B_1)$, $(X_3, Y_3, B_3)$ such that $(X_i, Y_i, B_i)_{i=1}^4\in \mathcal{F}$. 

We define the following set of 4-tuples that will describe factorization of one operator through the restriction of another.

$$\mathcal{P}=\{(X_i, Y_i, B_i)_{i=1}^4 \in \mathcal{L}^4: X_1=X_4,~ Y_1\subset X_2,~ B_2(Y_1)\subset X_3\subset Y_2,~Y_3=Y_4,~ B_3 B_2 B_1=B_4\}.$$

The next diagram illustrates the reason we care about $\mathcal{P}$.    

$$\begin{CD}  X_1=X_4 	@>B_4>>	Y_3=Y_4  \\ @VVB_1V			@AAB_3A \\ Y_1	@>B_2|_{Y_1}>>	X_3 \\ \cap & & \cap \\ X_2 @>B_2>> Y_3\end{CD}$$

That is, $B_4:X_4\to Y_4$ factors through a restriction of $B_2:X_2\to Y_2$ if and only if there exist $(X_1, Y_1, B_1)$, $(X_3, Y_3, B_3)$ such that $(X_i, Y_i, B_i)_{i=1}^4\in \mathcal{P}$.

Given $(X_i, Y_i, B_i)_{i=1}^4$, for any $1\leqslant i,j\leqslant 4$, $X_i\subset Y_j$, $X_i=Y_j$, $Y_j\subset X_i$ are Borel conditions in $\mathcal{L}^4$.  Furthermore, $B_2(Y_1)\subset X_3$ means that, if $B_2:X_2\to Y_2$ is the operator defined by $(X_2, Y_2, (y_n)_{n=1}^\infty)$, $y_n\in X_2$ for all $n\in \nn$, which is also a Borel condition.  By Proposition \ref{ddbb}, we deduce that the remaining conditions are Borel.

For $j=2,4$ let  $\pi_j:\mathcal{L}^4\to \mathcal{L}$, $\pi_j((X_i, Y_i, B_i)_{i=1}^4)= (X_j, Y_j, B_j)$ be the projection operator. Then each of the four classes can be realized as one of the four $$\pi_4(\mathcal{F}\cap [ \mathcal{L}\times \mathcal{B}\times \mathcal{L}^2]),$$ $$\pi_4(\mathcal{P}\cap [\mathcal{L}\times \mathcal{B}\times \mathcal{L}^2]),$$ $$\pi_2( \mathcal{F}\cap [\mathcal{L}^3\times \mathcal{B}]),$$ $$\pi_2(\mathcal{P}\cap [\mathcal{L}^3\times \mathcal{B}]).$$ 
This concludes the proof.
\end{proof}


For $(X,Y,A),(Z,W,B) \in \mathcal{L}$,  we write $(X,Y,A)\cong(Z,W,B)$ if $A$ is isomorphic to $B$.

\begin{lemma}
The set $\{((X,Y,A),(Z,W,B)) \in \mathcal{L}^2: A \cong B\}$
is analytic. Moreover if $\mathcal{A} \subset \mathcal{L}$ is analytic then 
$$\mathcal{A}_{\cong} =\{ (Z,W,B) \in \mathcal{L} : \exists (X,Y,A)\mbox{ with }(X,Y,A)\cong (Z,W,B)\}$$
is called the isomorphic saturation of $\mathcal{A}$ and is also analytic.
\label{keep em saturated}
\end{lemma}

\begin{proof} By \cite[Pages 10,11]{Do-Book}, the set of pairs $(X, (x_n)_{n=1}^\infty)\in \textbf{SB}\times C(2^\nn)^\nn$ such that $[x_n: n\in \nn]=X$ is Borel in $\textbf{SB}\times C(2^\nn)$, and the sets of pairs $((x_n)_{n=1}^\infty, (y_n)_{n=1}^\infty)\in C(2^\nn)^\nn\times C(2^\nn)^\nn$ such that there exists $k\in \nn$ such that $(x_n)_{n=1}^\infty$ and $(y_n)_{n=1}^\infty$ are $k$-equivalent is closed in $C(2^\nn)^\nn\times C(2^\nn)^\nn$.    We now consider the subset $\mathcal{B}$ of $(\mathcal{L}\times C(2^\nn)^\nn\times C(2^\nn)^\nn)^2$ consisting of those $$((X,Y,A), (x_n)_{n=1}^\infty, (y_n)_{n=1}^\infty, (Z,W,B), (z_n)_{n=1}^\infty, (w_n)_{n=1}^\infty)$$ such that there exist $k,l\in \nn$ such that \begin{enumerate}[(i)]\item $X=[x_n: n\in \nn]$, \item $Y=[y_n: n\in \nn]$, \item $Z=[z_n: n\in \nn]$, \item $W=[w_n: n\in \nn]$, \item $(x_n)_{n=1}^\infty$ is $k$-equivalent to $(z_n)_{n=1}^\infty$, \item $(y_n)_{n=1}^\infty$ is $l$-equivalent to $(w_n)_{n=1}^\infty$, \item for every $n\in \nn$, $Ax_n=y_n$ and $Bz_n=w_n$. \end{enumerate} Each of (i)-(vi) is a Borel condition by the previously cited fact from \cite{Do-Book}, and (vii) is a Borel condition by Lemma \ref{ddbb}. Therefore $\mathcal{B}$ is Borel. 

Now fix an analytic subset $\mathcal{A}$ of $\mathcal{L}$, a Polish space $P$, and a continuous function $f:P\to \mathcal{L}$ such that $f(P)=\mathcal{A}$.  Define $$g, \pi:P\times C(2^\nn)^\nn\times C(2^\nn)^\nn\times \mathcal{L}\times C(2^\nn)^\nn \times C(2^\nn)^\nn\to \mathcal{L}\times C(2^\nn)^\nn\times C(2^\nn)^\nn\times \mathcal{L}\times C(2^\nn)^\nn \times C(2^\nn)^\nn$$ and $$\Pi:\mathcal{L}\times C(2^\nn)^\nn\times C(2^\nn)^\nn\times \mathcal{L}\times C(2^\nn)^\nn\times C(2^\nn)^\nn\to \mathcal{L}\times \mathcal{L}$$ by $$g(x, \sigma, \tau, (Z,W,B), \sigma', \tau')= (f(x), \sigma, \tau, (Z,W,B), \sigma', \tau'),$$ $$\pi(x, \sigma, \tau, (Z,W,B), \sigma', \tau')=(Z,W,B),$$ and $$\Pi((X,Y,A), \sigma, \tau, (Z,W, B), \sigma', \tau')=((X,Y,A), (Z,W,B)).$$   Then $$\{((X,Y,A), (Z,W,B))\in \mathcal{L}^2: A\cong B\}= \Pi(\mathcal{B})$$ is analytic.  Furthermore, let $\mathcal{B}_0=g^{-1}(\mathcal{B})$ and note that $$\{(Z,W,B): (\exists (X,Y,A)\in \mathcal{A})((X,Y,A)\cong (Z,W,B))\}=\pi(\mathcal{B}_0)$$ is analytic.

\end{proof}


\subsection{The standard space $\mathcal{L}$ and the existence of universal operators}

Recall first that a class $\mathfrak{J}$ of operators is \emph{separably determined} if $A:X\to Y$ lies in $\mathfrak{J}$ if and only if $A|_Z:Z\to \overline{A(Z)}$ lies in $\mathfrak{J}$ for every separable subspace $Z$ of $X$. In this subsection, we use the complexity of the relations on $\mathcal{L}$ to show that whenever $\mathfrak{J}$ is a separably determined operator ideal so that  $\mathfrak{J}\cap \mathcal{L}$ is not coanalytic,   $\complement\mathfrak{J}$ does not admit a universal operator. There is an obvious problem that arises in using $\mathcal{L}$ to say anything about the non-existence of universal operators. In previous studies of universal operators, the operators were, in general, not restricted to having separable domain and range spaces.

\begin{lemma} Suppose $\mathfrak{J}$ is separably determined  operator ideal. If  $\complement \mathfrak{J}$ admits a universal factoring operator, then there exists $U\in \complement \mathfrak{J}\cap \mathcal{L}$ which factors through a restriction of every member of $\complement\mathfrak{J}\cap\mathcal{L}$.

Moreover, if we assume further that $\mathfrak{J}$ is injective, the converse holds.
\label{keep em separable}
\end{lemma}

\begin{proof}
Let $U':E' \to F'$ be $\complement \mathfrak{J}$-universal. Since $\mathfrak{J}$ is separably determined, there is a separable subspace $E$ of $E'$ so that $U'|_{E}:E \to F:=\overline{U'(E)}$ is in $\complement\mathfrak{J} \cap \mathcal{L}$. Let $U=U'|_{E}$. Then $U:E \to F$ is in $\mathcal{L}$ and it is easy see that $U$ factors through a restriction of every member of $\complement \mathfrak{J} \cap \mathcal{L}$.

Now we assume that $\mathfrak{J}$ is injective  and prove the converse. Let $U:X\to Y$ in $\complement \mathfrak{J} \cap \mathcal{L}$ and assume it factors through a restriction of every member of $\complement \mathfrak{J} \cap \mathcal{L}$. Fix  any isometric embedding $i:Y \to \ell_\infty$. Since $\mathfrak{J}$ is injective,  $iU \in \complement \mathfrak{J}$. Since $\ell_\infty$ is injective $iU$ factors through every operator in $\complement \mathfrak{J}$. Indeed, if $A:E\to F$ is in $\complement \mathfrak{J}$, then since $\mathfrak{J}$ is separably determined, there is a separable subspace $Z$ of $E$ such that $U$ factors through $A|_Z:Z \to \overline{A(Z}) \subset F$. Let $C:\overline{A(Z}) \to Y$ and $D:X \to Z$ be such that $U=CA|_{Z}D$. The map $iC$ can be extended to $\tilde{iC}:F \to \ell_\infty$ and $iU= \tilde{iC}AD$. This shows that $iU$ is universal for $\complement \mathfrak{J}$.
\end{proof}

The next theorem isolates the relationship between complexity of the ideal and the existence of universal operators.

\begin{theorem}
Let $\mathfrak{J}$ have the ideal property and suppose that $\mathfrak{J}\cap \mathcal{L}$ is not coanalytic. Then following hold:
\begin{enumerate}
    \item There is no Borel collection $\mathcal{B} \subset \complement \mathfrak{J}\cap \mathcal{L}$ such that for every $A \in \complement \mathfrak{J}\cap \mathcal{L}$, there is a $B\in \mathcal{B}$ such that $B$ factors through a restriction of $A$.
    \item If $\mathfrak{J}$ is separably determined then there is no $\complement \mathfrak{J}$-universal operator.
\end{enumerate}
\label{too complicated?}
\end{theorem}

\begin{proof}

Let $\mathfrak{J}$ have the ideal property so that $\mathfrak{J}\cap \mathcal{L}$ is not coanalytic.

We begin by proving item (1). Assume, towards a contradiction, that there is a Borel collection $\mathcal{B} \subset \complement  \mathfrak{J}\cap \mathcal{L}$, such that for every $A \in \complement \mathfrak{J}\cap \mathcal{L}$, there is a $B\in \mathcal{B}$ such that $B$ factors through a restriction of $A$. By Proposition \ref{lets have relations}(iii) the set of all operators $A \in \mathcal{L}$ so that there is a $B \in \mathcal{B}$ with $B$ factoring through a restriction of $A$ is analytic. On the other hand, using the ideal property of $\mathfrak{J}$, this set also coincides with the set $\complement \mathfrak{J}\cap \mathcal{L}$ which is assumed to not be analytic. 

Now assume $\mathfrak{J}$ is separably determined and, towards a contradiction, assume that $\complement \mathfrak{J}$ admits a universal factoring operator. Lemma \ref{keep em separable} yields that there is a $U \in \complement \mathfrak{J}\cap \mathcal{L}$ so that the collection of all operators $A \in \mathcal{L}$ so that $U$ factors through a restriction of $A$ is $\complement \mathfrak{J}\cap \mathcal{L}$. Again, by Proposition \ref{lets have relations}, this set is analytic. This is a contradiction.
\end{proof}



We will use Theorem \ref{too complicated?} to give an alternative proof of the result of Girardi and Johnson \cite{GirJo-Israel} that  $\complement \mathfrak{V}$ does not have a universal factoring operator. We also have the same result of the complement of the ideal of Dunford-Pettis operators $\mathfrak{DP}$. In order to proceed, we need the following theorem of O. Kurka \cite[Theorem 1.2]{Kurka-preprint}.

\begin{theorem}
The classes of all separable Banach spaces with the Schur property (=$\mbox{Space}(\mathfrak{V})$) and the Dunford Pettis property (=$\mbox{Space}(\mathfrak{DP})$) are $\Pi_1^2$-complete.
\label{Kurka}
\end{theorem}

\begin{theorem}
There is no universal operator for $\complement \mathfrak{V}$ or $\complement \mathfrak{DP}$. Moreover, there is no Borel subset of $\complement \mathfrak{V}\cap \mathcal{L}$ or $\complement \mathfrak{DP}\cap \mathcal{L}$ which  is universal for $\complement \mathfrak{V}\cap \mathcal{L}$ or $\complement \mathfrak{DP}\cap \mathcal{L}$, respectively.
\label{come at me bro}
\end{theorem}

\begin{proof}
Note that both the ideals we are considering are separably determined. Therefore by combining Theorem \ref{Kurka}, Lemma \ref{too hard}, and Theorem \ref{too complicated?}(2) we conclude that there is no universal operator for $\complement \mathfrak{V}$ or $\complement \mathfrak{DP}$. The second statement follows from Theorem \ref{too complicated?}(1).
\end{proof}

\subsection{Bossard-genericity and Strong Boundedness for subsets of $\mathcal{L}$}

In this subsection we prove the results on Bossard genericity and strong boundedness advertised at the beginning of the section.

In order to proceed, we isolate the following definition that generalizes the Bourgain embeddibility index for spaces \cite{Bou-PAMS}.

\begin{definition}
Let $(X,Y,B) \in \mathcal{L}$ and $\mathfrak{NP}_{B}$ be the subset of $\mathcal{L}$ of all operators $A$ so that $B$ does not factor through a restriction of $A$ (here $\mathfrak{NP}_{B}$ stands for ``not preserving $B$''). Fix $K >0$ and $(V,W,A)\in \mathcal{L}$ and let 
\begin{equation}
\begin{split}
T(A,B,K)=  \{ (v_i)_{i=1}^n\in V^{<\nn}:& \forall (a_i)_{i=1}^n,~   \|\sum_{i=1}^n a_i v_i\|\leqslant K\|\sum_{i=1}^n a_i d_i(X)\|, \\ 
&\|\sum_{i=1}^n a_i Bd_i(X)\|\leqslant K\|\sum_{i=1}^n a_i Av_i\|\}
\end{split}
\end{equation}
  We let $\textbf{NP}(A,B, K)=o(T(A,B, K))$ if this is a countable ordinal, and we let $\textbf{NP}(A,B, K)=\omega_1$ otherwise.  We let $$\textbf{NP}(A,B)=\sup_{K>0} \textbf{NP}(A,B, K).$$
  \end{definition}
  
  The next proposition is a standard fact which follows from the fact that $T(A,B,K)$ are closed trees in $V$.
  
\begin{proposition}
 $(X,Y,B) \in \mathcal{L}$ factors through a restriction of $(V,W,A) \in \mathcal{L}$ if and only if  $\textbf{NP}(A,B, K)=\omega_1$
\end{proposition}

 We note that using standard arguments \cite{Do-Book}, we have the following proposition.
 
 \begin{proposition}
Let $(X,Y,B) \in \mathcal{L}$. Then $\mathfrak{NP}_B$ is $\Pi_1^1$ and $A \mapsto \textbf{NP}(A,B)$ is a $\Pi_1^1$-rank on $\mathfrak{NP}_{B}$.
 \label{rank}
 \end{proposition}
 
In order to prove that a given collection of operators is Bossard generic we must define operators that have a large index with respect to the rank $A \mapsto \textbf{NP}(A,id_{C(2^\mathbb{N})})$, but which are members of the given collection. We will consider the following operators between James tree spaces. These spaces have appeared several times in the literature, perhaps most notably in \cite{ADo-Advances}.

\begin{definition}
Consider the following James type spaces and canonical operators between them. Let us say a subset $\mathfrak{s}$ of $[\nn]^{<\nn}$ is a \emph{segment} provided that there exist $s,t\in [\nn]^{<\nn}\setminus\{\varnothing\}$ such that $\mathfrak{s}=\{u: s\preceq u\preceq t\}$. We say two segments $\mathfrak{s}, \mathfrak{t}$ are \emph{incomparable} if for any $s\in \mathfrak{s}$ and $t\in \mathfrak{t}$, $s\not\preceq t$ and $t\not\preceq s$.  We recall that $|t|$ denotes the length of $t$. 

Let $(y_i)_{i=1}^\infty$ be a fixed Schauder basis for a Banach space $Y$.  For $1\leqslant p<\infty$, let  $\mathfrak{X}_p((y_i)_{i=1}^\infty)$ be the completion of $c_{00}([\nn]^{<\nn}\setminus\{\varnothing\})$ with respect to the norm $$\|\sum_{t\in \nn^{<\nn}} a_t e_t\| = \sup\Bigl\{\Bigl(\sum_{i=1}^n \|\sum_{t\in \mathfrak{s}_i} a_t y_{|t|}\|_Y^p\Bigr)^{1/p}: n\in \nn, \mathfrak{s}_1, \ldots, \mathfrak{s}_n\text{\ are pairwise incomparable segments}\Bigr\}.$$  For each tree $T$ on $\nn$, let $\mathfrak{X}^T_p((y_i)_{i=1}^\infty)$ denote the closed span of $\{e_t: t\in T\setminus \{\varnothing\}\}$ in $\mathfrak{X}_p((y_i)_{i=1}^\infty)$. 
\end{definition}

We would like to know that spaces $\mathfrak{X}^T_p((y_i)_{i=1}^\infty)$ for well-founded trees $T$ retain some desired property independent of the sequence $(y_i)$.

\begin{proposition} Fix $1\leqslant p< \infty$.  Assume also that ($P$) is a property such that: \begin{enumerate}[(i)]\item If $X$ has property ($P$) and $Z$ is isometrically isomorphic to $X$, then $Z$ has property ($P$). \item If $X$ is a Banach space and $Z\leqslant X$ is a subspace such that $\dim X/Z<\infty$ and $Z$ has ($P$), then $X$ has ($P$), \item If for each $n\in \nn$, $X_n$ is a Banach space with property ($P$), then $(\oplus_{n=1}^\infty X_n)_{\ell_p}$ has property ($P$), \item Every finite dimensional space has property ($P$). \end{enumerate}  Then for any well-founded tree $T$ on $\nn$  and any Schauder basis $(y_i)_{i=1}^\infty$, $\mathfrak{X}^T_p((y_i)_{i=1}^\infty)$ has property ($P$). 
\label{induct}
\end{proposition}

\begin{proof} We prove by induction on $\xi<\omega_1$ that if $o(T)\leqslant \xi+1$, $\mathfrak{X}^T_p((y_i)_{i=1}^\infty)$ has property ($P$). If $o(T)\leqslant 1$, then $T\subset \{\varnothing\}$ and  $\mathfrak{X}^T_p((y_i)_{i=1}^\infty)=\{0\}$  and therefore has property ($P$). 

Assume that $\xi$ is a limit ordinal and we have the result for every $\zeta<\xi$.  Fix a tree $T$ with $o(T)\leqslant \xi+1$.  Then for each $n\in \nn$, let $T_n=\{t\in T: (n)\preceq t\}$ and note that $o(T_n)<\xi$ for each $n\in \nn$.    From this it follows that $\mathfrak{X}^{T_n}_p((y_i)_{i=1}^\infty)$ has property ($P$), and so does $$\mathfrak{X}^T_p((y_i)_{i=1}^\infty) = (\oplus_{n=1}^\infty \mathfrak{X}^{T_n}_p((y_i)_{i=1}^\infty))_{\ell_p}.$$

This is the desired result.
\end{proof}

\begin{proposition} Fix a Schauder basis $(y_i)_{i=1}^\infty$,  a well-founded tree $T$ on $\nn$, and $1\leqslant p<\infty$. Then $\mathfrak{X}^T_1((y_i)_{i=1}^\infty)$ has the Schur property, $\mathfrak{X}^T_p((y_i)_{i=1}^\infty)$ is  $\ell_p$ saturated, and $\mathfrak{X}^T_2((y_i)_{i=1}^\infty)$, has the Banach-Saks property.  

In particular, the canonical inclusion $I^{1,2}_T:\mathfrak{X}^T_1((y_i)_{i=1}^\infty)\to \mathfrak{X}^T_2((y_i)_{i=1}^\infty)$ is completely continuous, weakly compact, and $\mathcal{S}_1$-$\mathfrak{S}$. 
\label{yeah okay}
\end{proposition}

\begin{proof}
The first claim follows from the fact that the Schur property satisfies the hypotheses of Proposition \ref{induct} when $p=1$.

The second claim follows from the fact that for any $1\leqslant p<\infty$, the property of being $\ell_p$ saturated satisfies the hypotheses of Proposition \ref{induct}. Here we remark that every finite dimensional space is vacuously $\ell_q$ saturated for any $1\leqslant q<\infty$.

The third claim follows from the fact that the Banach-Saks property satisfies the hypotheses of Proposition \ref{induct} when $1<p<\infty$.  
Since the identity on $\mathfrak{X}^T_1((y_i)_{i=1}^\infty)$ is completely continuous, so is $I_T^{1,2}$. Since the identity on $\mathfrak{X}^T_2((y_i)_{i=1}^\infty)$ has the Banach-Saks property,  so does $I^{1,2}_T$.  Finally, if $(x_i)_{i=1}^\infty \subset \mathfrak{X}^T_1((y_i)_{i=1}^\infty)$ is normalized and basic, then after passing to a subsequence, we may assume $(x_i)_{i=1}^\infty$ is equivalent to the $\ell_1$ basis. Then since $\mathfrak{X}^T_2((y_i)_{i=1}^\infty)$ has the Banach-Saks property, it does not admit an $\ell_1$ spreading model, which means $$\inf_{E\in \mathcal{S}_1} \min_{\|(a_i)_{i\in E}\|_{\ell_1}=1} \|I_T\sum_{i\in E} a_i x_i\|_{\mathfrak{X}_2^T((y_i)_{i=1}^\infty)}=0.$$ This yields that $I_T^{1,2}$ is $\mathcal{S}_1$-$\mathfrak{S}$. 
\end{proof}

The following is the proof of Theorem \ref{lots of stuff} and the generic statements in Theorem \ref{it's complicated}

\begin{proof}
Let $\mathfrak{J}:=\mathfrak{W}\cap \mathfrak{V} \cap \mathcal{S}_1\mbox{-}\mathfrak{S}$. We claim that $\mathfrak{J}$ is Bossard generic. This will give the desired result,  since $\mathfrak{J}$ is a subset of each of the classes from Theorem \ref{lots of stuff} and Theorem \ref{it's complicated}. Let $\mathcal{A} \subset \mathcal{L}$ be analytic and contain an isomorphic copy of every element of $\mathfrak{J}$. Then by Lemma \ref{keep em saturated} the isomorphic saturation $\mathcal{A}_{\cong}$ of $\mathcal{A}$ is analytic. Recall that our goal is to show that there is a $A \in \mathcal{A}$ so that every $B \in \mathcal{L}$, $B$ factors through a restriction of $A$. 

Towards a contradiction,  we assume that $\mathcal{A}_{\cong}$ is a subset of $\mathfrak{NP}_{id_{C(2^\mathbb{N})}}$. Since $\mathcal{A}_{\cong}$ is analytic, by Proposition \ref{rank} the map 
$$\mathcal{L}\ni (X,Y,A) \mapsto \mathbf{NP}(A,id_{C(2^\nn)})$$
is a $\Pi_1^1$-rank on $\mathfrak{NP}_{id_{C(2^\mathbb{N})}}$.
By Fact \ref{boundedness}, there exists $\zeta<\omega_1$ such that for each $B \in \mathcal{A}_{\cong}$, $\mathbf{NP}(B,id_{C(2^\mathbb{N})}) < \zeta$. Let $T_\zeta$ be a tree on $\mathbb{N}$ with $o(T_\zeta)> \zeta$. 
Consider the operator $I^{1,2}_{T_\zeta}:\mathfrak{X}^{T_{\zeta}}_1((e_i)_{i=1}^\infty)\to \mathfrak{X}^{T_{\zeta}}_2((e_i)_{i=1}^\infty)$ where $(e_i)$ is a Schauder basis of $C(2^\mathbb{N})$. By Proposition \ref{yeah okay}, the operators $I^{1,2}_{T_\zeta}$ is in the ideal $\mathfrak{W}\cap \mathfrak{V} \cap \mathcal{S}_1\mbox{-}\mathfrak{S}$, and therefore any realization of $I^{1,2}_{T_\zeta}:\mathfrak{X}^{T_\zeta}_1((y_i)_{i=1}^\infty\to \mathfrak{X}^{T_\zeta}_2((y_i)_{i=1}^\infty)$ as a member of $\mathcal{L}$  is in $\mathcal{A}_{\cong}$. Here we use that we are considering the isomorphic saturation of $\mathcal{A}$ and just just $\mathcal{A}$. On the other hand, $\mathbf{NP}(I^{1,2}_{T_\zeta},id_{C(2^\mathbb{N})})>\zeta$. Indeed, for any $t\in {T_{\zeta}}$, consider $(e_s)_{\varnothing\prec s\preceq t}\subset c_{00}(\nn^{<\nn}\setminus \{\varnothing\})$.   
Then since for any pairwise incomparable segments $\mathfrak{s}_1, \ldots, \mathfrak{s}_n$, $\mathfrak{s}_i\cap \{u: \varnothing\prec u\preceq t\}\neq \varnothing$ for at most one value of $i$, $$\|\sum_{i=1}^{|t|} a_i e_i\| \leqslant \|\sum_{\varnothing\prec s\preceq t} a_{|s|}  e_s\|_{\mathfrak{X}^{T_{\zeta}}_p((e_i)_{i=1}^\infty)}  \leqslant 2b \|\sum_{i=1}^{|t|} a_i e_i\|$$ for any $1\leqslant p<\infty$. Here, $b$ is the basis constant of $(e_i)_{i=1}^\infty$ in $C(2^\nn)$. 
This shows that the map $t\mapsto ((2b)^{-1}e_s)_{\varnothing\prec s\preceq t}$, $\varnothing\mapsto \varnothing$ is a tree isomorphism of ${T_{\zeta}}$ into $$T_{\mathfrak{S}C(2^\nn)}(I^{1/2}_{T_\zeta},\mathfrak{X}^{T_{\zeta}}_1((y_i)_{i=1}^\infty),\mathfrak{X}^{T_{\zeta}}_2((y_i)_{i=1}^\infty), 2b),$$ 
and $r_{\mathfrak{S}C(2^\nn)}(\mathfrak{X}^{T_{\zeta}}_1((y_i)_{i=1}^\infty),\mathfrak{X}^{T_{\zeta}}_2((y_i)_{i=1}^\infty), I^{1,2}_{T_\zeta})>\zeta$. This contradiction yields the desired result.
\end{proof}

We can now prove our spatial results.

\begin{proof}[Proof of Theorem \ref{spatial}]
Let $\textbf{C}$ be any one of the classes $\textbf{NC}_E$, $\textbf{DP}$, $\textbf{V}$ or $\textbf{BS}$. If $\textbf{C}=\textbf{DP}$ or $\textbf{C}=\textbf{V}$, let $p=1$. If $\textbf{C}=\textbf{BS}$, let $p=2$. If $\textbf{C}=\textbf{NC}_E$, fix any $1<p<\infty$ such that $E$ is not $\ell_p$ saturated.  

Let $(e_i)$ be a Schauder basis of $C(2^\mathbb{N})$ and note that for any well-founded tree $T$ on $\nn$, $\mathfrak{X}^T_p((e_i)_{i=1}^\infty)$ lies in $\textbf{C}$ by Proposition \ref{yeah okay}.  We argue as in the previous theorem that  $$r_{\mathfrak{S}C(2^\nn)}(\mathfrak{X}^{T}_p((e_i)_{i=1}^\infty), \mathfrak{X}_p^T((e_i)_{i=1}^\infty), I_{\mathfrak{X}_p^T((e_i)_{i=1}^\infty)})\geqslant o(T).$$  Since we may find well-founded trees $T$ on $\nn$ with arbitrarily large, countable order, we deduce by a standard overspill argument that any separable space containing isomorphs of all $\mathfrak{X}_p^T((e_i)_{i=1}^\infty)$, $T\subset \nn$ well-founded, must have uncountable $r_{\mathfrak{S}C(2^\nn)}$ index and therefore contain an isomorph of $C(2^\nn)$. 
\end{proof}


We now observe that Bossard-genercity of a collection of space implies that corresponding operator ideal is Bossard generic. This yields an alternative proof for the spatial operators ideals found in Theorem \ref{lots of stuff}.

\begin{proposition}
Let $\mathfrak{J} \subset \mathcal{L}$ have the ideal property and suppose that $\mbox{Space}(\mathfrak{J})=\{X \in \textbf{SB}:(X,X,(d_n(X)))\in \mathfrak{J}\}$ is Bossard-generic in $\textbf{\emph{SB}}$. Then $\mathfrak{J}$ is Bossard-generic in $\mathcal{L}$
\label{weak sauce}
\end{proposition}

\begin{proof}
Let $\mathcal{A}\subset \mathcal{L}$ be analytic and contain an isomorphic copy of every element of  $\mathfrak{J}$. Recall that $\Psi:\textbf{SB} \to \mathcal{L}$ defined by $\Psi(X)=(X,X,(d_n(X)))$ is a Borel isomorphism and $\textbf{Sp}:=\Psi(\textbf{SB})$. Consider the isomorphic saturation $\mathcal{A}_{\cong}$ of $\mathcal{A}$. Lemma \ref{keep em saturated}, yields that $\mathcal{A}_{\cong}$ is analytic. Then $\mathcal{A}_{\cong} \cap \textbf{Sp}$ is analytic in $\mathcal{L}$ and $\Psi^{-1}(\mathcal{A}_{\cong} \cap \textbf{Sp})$ is an analytic subset of $\textbf{SB}$ containing $\mbox{Space}(\mathfrak{J})$. As $\mbox{Space}(\mathfrak{J})$ is assumed to be Bossard-generic, there is an $X \in \Psi^{-1}(\mathcal{A}_{\cong} \cap \textbf{Sp})$ which contains all separable Banach spaces, and $(X,X,(d_n(X)))\in \mathcal{A}_{\cong} \cap \textbf{Sp}$. Since  all separable Banach spaces embed in $X$, then every element of $\mathfrak{J}$ factors throught a restriction of $(X,X,(d_n(X)))$. This is the desired result.
\end{proof}

 \begin{remark} 
In general it is not true that $\text{Space}(\mathfrak{J})$ is Bossard generic whenever $\mathfrak{J}$ is Bossard generic. Indeed the strictly singular operators serve as a counterexample.
\end{remark}




We conclude this subsection with the proofs of our strongly bounded results.

\begin{proof}[Proof of Theorem \ref{strongly bounded}]
Let $\mathcal{A}$ be an analytic subset of $\mathfrak{W}$. In \cite{C-Fund}, the second author defined a $\Pi_1^1$-rank $r_{\mathfrak{W}}$ on $\mathfrak{W}$. For each $\xi<\omega_1$ let 
$$\mathscr{J}_{\omega^{\omega^\xi}}=\{(X,Y,A) \in \mathcal{L} : r_{\mathfrak{W}}(X,Y,A) \leqslant \omega^{\omega^\xi}\}.$$
Since $\mathcal{A}$ is analytic and $r_{\mathfrak{W}}$ is a $\Pi_1^1$-rank on $\mathfrak{W}$ there is a $\xi<\omega_1$ such that $\mathcal{A}\subset \mathscr{J}_{\omega^{\omega^\xi}}$. The following properties are satisfied for $\mathscr{J}_{\omega^{\omega^\xi}}$.
\begin{enumerate}
    \item $\mathscr{J}_{\omega^{\omega^\xi}}$ is a closed operator ideal \cite{C-Fund}.
    \item For each $(X,Y,A)\in \mathscr{J}_{\omega^{\omega^\xi}}$ there is a $Z \in \mbox{Space}(\mathscr{J}_{\omega^{\omega^{\xi+1}}})$ such that $(X,Y,A)$ factors through $Z$ \cite{BC-Scand}.
\end{enumerate}
The set $\textbf{Sp} \cap\mathscr{J}_{\omega^{\omega^{\xi+1}}} \subset \mathcal{L}$ is Borel. Therefore $\Psi^{-1}(\textbf{Sp} \cap\mathscr{J}_{\omega^{\omega^{\xi+1}}})$ is Borel in $\textbf{SB}$ and a subset of $\textbf{REFL}$. Since $\textbf{REFL}$ is strongly bounded there is a $W\in \textbf{REFL}$ so that each $Z \in \Psi^{-1}(\textbf{Sp} \cap\mathscr{J}_{\omega^{\omega^{\xi+1}}})$ embeds in $W$. 

Fix $(X,Y,A) \in \mathcal{A}$. We claim that $A$ factors through a restriction of $id_{W}$. By $(2)$, $A$ factors through some $Z' \in\mbox{Space}(\mathscr{J}_{\omega^{\omega^{\xi+1}}})$. But, this $Z'$ can be identified isometrically with  an element of $Z \in \Psi^{-1}(\textbf{Sp} \cap\mathscr{J}_{\omega^{\omega^{\xi+1}}})$ which embedds isomorphically into $W$. This implies that $A$ factors through a restriction of $id_{W}$.

The case of $\mathfrak{D}$ follows from a similar argument. In this case, we use several results of Brooker \cite{Br-JOT} together with the fact that $\textbf{SD}$ is strongly bounded. Namely, we note that the Szlenk index is a coanalytic rank on $\mathfrak{D}$ and that Brooker showed that every $A \in \mathfrak{D}_\xi$ factors through a space $X \in \mbox{Space}(\mathfrak{D}_{\xi+1})$.  \end{proof}







\subsection{Borel Ideals} In this final subsection we give several examples of ideals that are Borel as subset of $\mathcal{L}$. The proofs are standard. 


Many notable operator ideals are defined as certain factorizations of sequences of operators between Banach spaces. Suppose $E$ is a finite dimensional Banach space with basis $(e_i)_{i=1}^n$ and $B:E\to F$ is an operator.  For $k\in \nn$, natural numbers $p_1, \ldots, p_n$, and rational numbers $q_1, \ldots, q_n$, we let $F_B(k, p_1, \ldots, p_n;q_1, \ldots, q_n)$  denote the set of triples $(X,Y, A)\in \mathcal{L}$ such that $\|\sum_{i=1}^n q_i d_{p_i}(X)\|\leqslant k \|\sum_{i=1}^n q_i e_i\|$ and $\text{resp.\ }\|\sum_{i=1}^n q_i Be_i\|\leqslant k \|\sum_{i=1}^n q_i Ad_{p_i}(X)\|.$       It is clear that these are Borel sets.  Now if for each $s\in \nn$, $B_s:E_s\to F_s$ is an operator and $\dim E_s=l_s$, then the collection of all operators in $\mathfrak{L}$ through which the collection $(B_s)_{s=1}^\infty$ uniformly factors is $$\bigcup_{k\in \nn}\bigcap_{s\in \nn} \bigcup_{(p_1, \ldots, p_{l_s})\in \nn^{l_s}} \bigcap_{(q_1, \ldots, q_{l_s})\in \mathbb{Q}^{l_s}} F_{B_s}(k, p_1, \ldots, p_{l_s};q_1, \ldots, q_{l_s}),$$  which is Borel. With this, we note that uniform factorization of a sequence of  operators between finite dimensional spaces gives rise to Borel classes. From this we deduce the following.

\begin{corollary} Each of the following classes of operators is Borel in $\mathcal{L}$. \begin{enumerate}[(i)]\item The compact operators. \item The super weakly compact operators. \item The finitely strictly singular operators. \item The super Rosenthal operators.  \end{enumerate}
\end{corollary}

We remark that for the super weakly compact, finitely strictly singular,  and super Rosenthal operators, it is already known that these are Borel sets.  Indeed, this above follows from  the fact that each of these three classes is the set of operators for which a particular coanalytic rank does not exceed $\omega+1$. This was shown in  \cite{BCFrWa-JFA,C-Fund}.

Furthermore, any local property of an operator (that is, any property which is determined by checking some condition on only  finitely many vectors at a time) gives rise to a Borel class of operators, again by standard arguments. In particular, we have the following.

\begin{proposition} For any $1\leqslant p \leqslant \infty$, each of the following classes is Borel in $\mathfrak{L}$. \begin{enumerate}[(i)]\item The operators with Rademacher/Gaussian/Haar/martingale type $p$ (resp. cotype $p$). \item The uniformly convex (resp. $p$-convex) operators. \item The uniformly smooth (resp. $p$-smooth) operators. \item The absolutely $p$-summing operators. \end{enumerate}
\end{proposition}

\begin{proof} Since the proof consists of standard techniques and is similar for each of the above classes, we only offer a proof for the absolutely $p$-summing operators and leave it to the reader to fill in the details of the other classes.  We note that the subclass of $\mathcal{L}$ consisting of absolutely $p$-summing operators is given by  $$\bigcup_{k\in \nn} \bigcap_{n=1}^\infty \bigcap_{(a_i)_{i=1}^n\in \mathbb{Q}^n}  \{(X,Y, (y_i)_{i=1}^\infty): \|\sum_{i=1}^n y_i\|\leqslant k\|(a_i)_{i=1}^n\|_{\ell_q^n} \|\sum_{i=1}^n a_i d_i(X)\|\},$$ 
\end{proof} where $1/p+1/q=1.$

\bibliographystyle{abbrv}

\def\cprime{$'$} \def\cprime{$'$} \def\cprime{$'$} \def\cprime{$'$}

\end{document}